\documentclass[12pt,a4paper,reqno]{amsart}
\usepackage[english]{babel}
\usepackage{amsmath}
\usepackage{amsfonts}
\usepackage{amsthm}
\usepackage{amssymb}
\usepackage[T1]{fontenc}
\usepackage{ae}
\usepackage{epsfig}
\usepackage{graphicx}
\usepackage{times}
\usepackage{latexsym}
\usepackage{xspace}
\usepackage{color}
\usepackage[colorlinks=true,linkcolor=red,citecolor=blue]{hyperref}
\usepackage{accents}
\theoremstyle{plain}

\newtheorem{theo}{Theorem}[section]
\newtheorem{pro}[theo]{Proposition}
\newtheorem{lem}[theo]{Lemma}

\newtheorem{cor}[theo]{Corollary}

\theoremstyle{remark}
\newtheorem{remark}[theo]{Remark}

\numberwithin{equation}{theo}

\theoremstyle{teorema}
\newtheorem*{Firmo}{Theorem}
\newtheorem*{teor01}{Theorem (Plante-Thurston)}
\newtheorem*{teor02}{Theorem (Farb-Franks)}
\newtheorem*{teor03}{Theorem (Farb-Franks)}
\newtheorem*{teor04}{Theorem (Navas)}

\theoremstyle{definition}

\newtheorem*{defn*}{Definition}

\newcommand{\Z}{\mathbb{Z}}

\newcommand{\R}{\mathbb{R}}
\newcommand{\RR}{\mathbb{R}^2}

\newcommand{\D}{\mathcal{D}}

\newcommand{\Ss}{\mathbb{S}^{2}}

\newcounter{hours}%
\newcounter{minutes}%

\begin{document}

\thispagestyle{empty}

\title[Fixed points for nilpotent actions on the plane]{Fixed points for nilpotent actions on the plane\\ and the
Cartwright\,-\,Littlewood theorem}
\author{S. Firmo, J. Rib\'on}
\address[S. Firmo and J. Rib\'on]{Instituto de Matem\'atica e Estat\'\i stica \\
Universidade Federal Fluminense, Rua M\'ario Santos Braga s/n - Valonguinho, 24020\,-140 Niter\'oi, RJ - Brasil}

\email{firmo@mat.uff.br}
\email{javier@mat.uff.br}

\author{J. Velasco}
\address[J. Velasco]{Instituto de Matem\'atica e Estat\'\i stica \\
Universidade Estadual do Rio de Janeiro, Rua S\~ao Francisco Xavier 524, 20550\,-\,013  Rio de Janeiro, RJ - Brasil}
\email{jaime@mat.uff.br}

\date{}
\subjclass{37C25, 37E30, 37E45}
\keywords{fixed point, recurrence, nilpotent group, homeomorphism, invariant continuum, winding number}
\thanks{Supported in part by CAPES}

\begin{abstract}
The goal of this paper is proving the existence and then localizing global fixed points for nilpotent groups
generated by homeomorphisms of the plane satisfying a certain Lipschitz condition.
The condition is inspired in a classical result of Bonatti for commuting diffeomorphisms of
the $2$-sphere and in particular it is satisfied by diffeomorphisms, not necessarily of class $C^{1}$,
whose linear part at every point is uniformly close to the identity.
In this same setting we prove a version of the Cartwright-Littlewood theorem,
obtaining fixed points in any continuum preserved by a nilpotent action.
\end{abstract}

\maketitle

\thispagestyle{empty}

\section{Introduction}
In 1989 Bonatti proved that two commuting diffeomorphisms of the sphere
\,${\mathbb S}^{2}$ have a common fixed point if they are $C^{1}$-close to the
identity \cite{bo01}. The result still holds true if the diffeomorphisms are
only $C^{0}$-close to the identity \cite[Handel, 1992]{han01}.
The case of commuting diffeomorphisms can be reinterpreted as the study of
a \,${\mathbb Z}^{2}$-action. In this context Druck-Fang-Firmo
(cf. \cite[2002]{dff02}) generalized Bonatti's theorem for nilpotent actions on the sphere.

Let us consider the plane, i.e. the other example of simply connected surface.
There are orientation-preserving homeomorphisms \,$f$\, of the plane that have no fixed points,
but Brouwer's translation theorem \cite[1912]{bro01} implies
\,$\lim_{n \to \infty} \|f^{n}(z)\|= \infty$\, for any \,$z \in {\mathbb R}^{2}$.
In particular an orientation-preserving homeomorphism of the plane that
preserves a non-empty compact set has a fixed point. In the same spirit
the proof of a theorem of Lima \cite[1964]{elon02}, that provides common singular points for
finite dimensional abelian Lie algebras of vector fields in \,${\mathbb S}^{2}$,
can be adapted naturally to show that if
a topological action of the additive group \,${\mathbb R}^{n}$\, on \,${\mathbb R}^{2}$\,
preserves a non-empty compact set then it has a global fixed point,
i.e. a common fixed point for all homeomorphisms in the action.
Moreover, Lima's original theorem was generalized by Plante \cite[1986]{pl01}:
Every topological action of a connected finite dimensional
nilpotent Lie group on ${\mathbb S}^{2}$
has a global fixed point.

Franks-Handel-Parwani \cite[2007]{fhp01}
extend Brouwer's property in the setting of discrete abelian groups.
More precisely, they
show that a finitely generated abelian subgroup \,$G$\,
of \,$\mathrm{Diff}_{+}^{1}({\mathbb R}^{2})$\, (orientable $C^{1}$-diffeomorphisms)
preserving a non-empty compact set has a global fixed point.
Such result is a fundamental ingredient in their characterization of
abelian groups of diffeomorphisms of \,${\mathbb S}^{2}$\, having a global fixed point.
The first author localized the fixed point if the group is generated by
$C^{1}$-close to the identity diffeomorphisms \cite[2011]{fi05}.
Indeed there exists a fixed point in the convex hull of the closure
of every bounded $G$-orbit.

\vglue10pt

\begin{Firmo}
Let \,$G$\, be an abelian group generated by elements of
$$\big\{ f \in \mathrm{Diff}^1(\RR) \ ; \ \|f(x)-x\| \ , \ \|Df(x)-Id\| < \epsilon
\ \ \mathrm{in} \ \ \RR \big\}. $$
Suppose that there exists a point \,$p \in \RR$ whose $G$-orbit is bounded.
Then there exists \,$q \in \mathrm{Fix}(G) \cap \mathrm{Conv}(\overline{\mathcal{O}_p(G)})$\,,
when \,$\epsilon>0$\, is small enough.
\end{Firmo}

We denote by \,${\mathcal{O}_p(G)}$\, the orbit of the point \,$p$\, by the group \,$G$.
We denote by   \,$\mathrm{Int}(A)$,
\,$\overline{A}$\, and \,$\mathrm{Conv}(A)$\, the interior, the closure and the convex
hull respectively of a subset \,$A$\, of \,${\mathbb R}^{2}$.
The notation \,$\mathrm{Fix}(G)$\, stands for the set of common fixed points of all elements of \,$G$.
The norm \,$\| \cdot \|$\, defined in \,${\mathbb R}^{2}$ is the norm associated to
the canonical inner product.

The goal of this paper is providing a generalization of the previous theorem
for nilpotent actions
under weaker hypotheses on the generators of \,$G$. More precisely, the generators of \,$G$\, are not required
to be $C^{1}$-diffeomorphisms but just homeomorphisms satisfying a certain
\,{\it Lipschitz condition}.
The Lipschitz condition guarantees the existence of fixed points in the convex hull
of the closure of every bounded $G$-orbit.
The next result is the main theorem of the paper\,:

\vskip10pt

\begin{theo}\label{teorema}
For any \,$\sigma \in {\mathbb Z}^{+}$\, there exists \,$\delta_{\sigma} \in {\mathbb R}^{+}$\,
such that$\,:$
\,If \,$G$\, is a \,$\sigma$-step nilpotent group generated by a family of
\,$\delta_{\sigma}$-Lipschitz with respect to the identity homeomorphisms of the plane
and there exists \,$p \in {\mathbb R}^{2}$ whose $G$-orbit is bounded then there exists
\,$q \in \mathrm{Fix}(G)$\, that belongs either to
\,$\overline{\mathcal{O}_p(G)}$\, or to \,$\mathrm{Int}\big(\mathrm{Conv}(\mathcal{O}_p(G))\big)$.
\end{theo}

Given a group \,$G$, we define inductively
$$G_{(0)}:=G  \quad \mathrm{and} \quad G_{(i+1)}:= [\,G\,,G_{(i)}\,]  $$
where \,$[\,G\,,G_{(i)}\,]$\, is the subgroup of \,$G$\, generated by the commutators of the form
\,$[\,a\,,b\,]:= aba^{-1}b^{-1}$\, with \,$a\in G$\, , \,$b\in G_{(i)}$\, and \,$i\geq 0$\,.
The groups \,$(G_{(j)})_{j \geq 0}$\, are the elements of the \,{\it lower central series}\, of \,$G$.
They are characteristic subgroups of \,$G$\, and in particular normal.
If \,$G_{(\sigma)}$\, is the trivial group for some \,$\sigma \in {\mathbb Z}_{\geq 0}$\,,
we say that \,$G$\, is a \,{\it nilpotent group}. The smallest \,$\sigma \in {\mathbb Z}_{\geq 0}$\, such that
\,$G_{(\sigma)}=\{Id\}$\, is the \,{\it nilpotency class}\, of \,$G$\,;
we say that \,$G$\, is a \,{\it $\sigma$-step nilpotent group}.
Clearly a non-trivial group \,$G$\, is $1$-step nilpotent if and only if it is abelian.

A map \,$g:\RR \rightarrow \RR$\, is \,$k$-{\it Lipschitz}\, if
\,$\|g(p)-g(q)\|\leq k \, \|p-q\|$\, in \,$\RR$.
The map \,$g$\, is \,{\it Lipschitz}\, if it is $k$-Lipschitz for some \,$k \geq 0$.
In such a case we denote by  \,$\mathrm{Lip}(g)$\, the smallest \,$k \geq 0$\,
such that \,$g$\, is $k$-Lipschitz.
We say that
\,$f:\RR\rightarrow \RR$\, is \,{\it $k$-Lipschitz with respect to the identity}\, if
\,$f-Id$\, is $k$-Lipschitz.
This property is satisfied by any differentiable map \,$f:{\mathbb R}^{2} \to {\mathbb R}^{2}$,
not necessarily of class $C^{1}$, with \,$||Df (x) - Id|| \leq k$\,
in \,${\mathbb R}^{2}$.

The maps \,$f:\RR \rightarrow \RR$\, such that \,$\mathrm{Lip}(f-Id)<1$\, are orientation-preserving
homeomorphisms of the plane (Lemma \ref{elondifeo})
that are isotopic to the identity by the barycentric isotopy
\,$F_{t}(x):= t f(x) +(1-t) x$\, (Corollary \ref{isotopia}).
Notice that \,$\{F_{t}\}_{t\in[0,1]}$\, is an isotopy issued from the identity relative to
\,$\mathrm{Fix}(f)$.
This reminds the techniques used by Franks-Handel-Parwani \cite{fhp01, fhp02}
to find global fixed points of abelian actions by $C^{1}$-diffeomorphisms on surfaces.
Indeed a $C^{1}$ orientation-preserving diffeomorphism \,$f$\,
satisfies that \,$f$\, is isotopic to the identity relative to \,$\mathrm{Fix}(f)$\,
in some neighborhood of the accumulation points of \,$\mathrm{Fix}(f)$\,
\cite[Lemma 3.8]{fhp02}. This property is crucial to obtain a well-behaved
Thurston decomposition of the elements of the abelian group.
The $C^{1}$ condition can be interpreted as a local uniformity property.
In our context the
Lipschitz condition plays an analogous role, but in contrast, our uniform condition
is of global nature.

Returning to the Brouwer hypothesis, Cartwright-Littlewood proved that an
orientation-preserving homeomorphism of the
plane possessing an invariant full continuum \,${\mathcal C}$\,
has a fixed point in \,${\mathcal C}$\, \cite[1951]{cl01}.
A \,{\it continuum}\, is by definition a non-empty connected compact subset \,${\mathcal C}$\, of \,$\RR$.
It is a \,{\it full continuum}\, if additionally it satisfies that \,$\RR - {\mathcal C}$\, is connected.
It is \,$G$-{\it invariant}\, if
\,$g({\mathcal C})={\mathcal C}$\, for any \,$g \in G$.
Three years later Hamilton (cf. \cite{hami01}) publishes a \,\emph{short proof}\,
of this result. Brown (cf. \cite[1977]{br05}) gives a  \,\emph{short short proof}\, of Cartwright-Littlewood
theorem, showing that it is a simple consequence of Brouwer's theorem.

There exists another short proof of Cartwright-Littlewood theorem
that has been in general omitted from the references in the literature.
As Guillou (cf. \cite[2012]{gui01}) explains in a recent paper in his
\,{\it Historical remark 1.4}\,:
\,{\it \ldots \ four years later, in the same Annals, Reifenberg
\cite{reif01} explained a short elementary proof due to Brouwer of the same result}\,.

Our proof of Cartwright-Littlewood's theorem, in the $\epsilon$-Lipschitz with respect
to the identity context, is inspired by Brown's idea: existence of fixed points implies localization
in an invariant full continuum. The proof also reminds Reifenberg's arguments since
it relies on calculating some winding numbers associated to certain closed curves.

\begin{theo}\label{CLnilp}
Let \,$G$\, be a $\sigma$-step nilpotent group generated by a family of \,$\delta_{\sigma}$-Lipschitz
with respect to the identity homeomorphisms of the plane, where
\,$\delta_{\sigma}$\, is provided in  Theorem \ref{teorema}.
If \,$\mathcal{C}$\, is a $G$-invariant full continuum then there exists a global fixed point of \,$G$\,
in \,${\mathcal C}$.
\end{theo}%

In 2009, Shi-Sun proved \cite{ss01} that every topological action of a nilpotent group on a
uniquely arcwise connected continuum has a global fixed point.
By definition a \,{\it uniquely arcwise connected continuum}\, is a non-empty, compact, connected, metrizable space \,$X$\,
such that any pair of points is connected by a unique path in \,$X$.

Recently, after the completion of this work, the second author proved a version of
Franks-Handel-Parwani's theorem for nilpotent subgroups
of \,$\mathrm{Diff}_{+}^{1}({\mathbb R}^{2})$\, and \,$\mathrm{Diff}_{+}^{1}({\mathbb S}^{2})$\,
(cf. \cite[2012]{ribon01}).

\vglue10pt

\noindent
{\it Acknowledgments}\,.
It is with great pleasure that we thank M\'{a}rio Jorge Dias Carneiro for
suggesting us to replace the condition of derivative
close to the identity in the first version by the actual Lipschitz condition.

\vskip30pt
\section{Outline of the proof of Theorem \ref{teorema}}
\vskip5pt

Let \,$G$\, be a $\sigma$-step  nilpotent group generated by a family
 \,$\mathcal{S}$\, of  \,$\delta_{\sigma}$-Lipschitz
 with respect to the identity homeomorphisms of the plane.
Let \,$p \in {\mathbb R}^{2}$\, be a point whose $G$-orbit is a bounded set.
The constant \,$\delta_{\sigma}$\, that only depends on the nilpotency class \,$\sigma$\,
will be determined by a series of technical lemmas in section
\ref{ap:a} and appendix \ref{ap:b} about the properties of
homeomorphisms Lipschitz with respect to the identity map.

Our goal is proving Theorem \ref{teo:mai} that is equivalent to
Theorem \ref{teorema}.
Suppose for simplicity that \,$\mathcal{S}=\{f_{1},\ldots,f_{n}\}$\,
is a finite family.
Roughly speaking, the starting point in order to prove Theorem \ref{teo:mai}
is considering that there exists \,$p \in \mathrm{Fix}(G_{(1)})$\,
whose $G$-orbit is bounded.
Such a hypothesis guarantees that the action of \,$G$\, on \,$\overline{{\mathcal O}_{p}(G)}$\,
is by pairwise commuting homeomorphisms.
Then we find a point \,$q_{1} \in \mathrm{Fix}(G_{(1)},f_{1})$\,
such that
\,$\mathcal{O}_{q_{1}}(G) \subset \mathrm{Conv}(\overline{\mathcal{O}_{p}(G)})$\,
(Lemma \ref{lem:clcv} and Proposition \ref{pro:im}).
Analogously we obtain a sequence of points \,$q_{0}=p$\ ,
\,$q_{j} \in \mathrm{Fix}(G_{(1)},f_{1},\hdots,f_{j})$\, for \,$1 \leq j \leq n$\, such that
\,$\mathcal{O}_{q_{j+1}}(G) \subset \mathrm{Conv}(\overline{\mathcal{O}_{q_{j}}(G)})$\,
for any \,$0 \leq j < n$. We can fine-tune the method to show that either
\,$\mathcal{O}_{q_{j+1}}(G) \subset \overline{\mathcal{O}_{q_{j}}(G)}$\, or
\,$\mathrm{Conv}(\overline{\mathcal{O}_{q_{j+1}}(G)}) \subset \mathrm{Int}(\mathrm{Conv}({\mathcal{O}_{q_{j}}(G)}))$\,
for any \,$0 \leq j < n$.
In particular \,$q_{n}$\, is a global fixed point of \,$G$\,
that belongs to
\,$\overline{{\mathcal O}_{p}(G)}$\, or to \,$\mathrm{Int}(\mathrm{Conv}({\mathcal{O}_{p}(G)}))$.
Let us remark that some of the intermediate results in the proof of Theorem
\ref{teo:mai} are generalizations to the Lipschitz setting of results in \cite{bo01,dff02,fi05}\,.

Given a group \,$G$\, and a subset \,$\mathcal{S}$\, of \,$G$\, we define
\,$\langle \mathcal{S} \rangle$\, as the subgroup of \,$G$\, generated by
\,$\mathcal{S}$.
We use in the proof that the class of \,$f_{j+1}$\, in
\,$G/ \langle G_{(1)},f_{1},\hdots,f_{j} \rangle$\, belongs to the center.
Indeed we can replace the roles of \,$\langle G_{(1)},f_{1},\hdots,f_{j} \rangle$\, and \,$f_{j+1}$\,
by a normal subgroup \,$H$\, and an element  \,$f$\, whose class in \,$G/H$\, belongs
to the center of \,$G/H$, denoted by \,$Z(G/H)$.
More precisely, given \,$q \in \mathrm{Fix}(H)$\, whose $G$-orbit is bounded there
exists \,$q' \in \mathrm{Fix}(H,f)$\, such that
\,$\mathcal{O}_{q'}(G) \subset \mathrm{Conv}(\overline{\mathcal{O}_{q}(G)})$\,
(Proposition \ref{pro:im}).
Theorem \ref{teo:mai} is a consequence of the fact that any finitely generated nilpotent group
is a tower of cyclic central extensions.


\vskip30pt
\section{Properties of the \,$\epsilon$-Lipschitz with respect to the identity
homeomorphisms}
\label{ap:a}
\vglue10pt


In this section we exhibit some properties of the maps
\,$f:{\mathbb R}^{2} \to {\mathbb R}^{2}$\,
that are \,$\epsilon$-Lipschitz with respect to the identity.
If \,$0< \epsilon<1$\, then \,$f$\, is a homeomorphism that is isotopic
to the identity by the barycentric isotopy.
Moreover the image of a line segment \,$[\,p\,,q\,]$\, by \,$f$\, is
contained in the closed ball of center \,$f(p)$\, and radius \,$\|f(p)-f(q)\|$\,
if \,$\epsilon < 1/(1 +\sqrt{3})$.

\vskip5pt

In order to show that \,$\mathrm{Lip}(f-Id)<1$\, implies that \,$f^{-1}$\, is well-defined
and continuous we use the following lemma (cf. \cite[p. 49]{shub01}).


\vskip10pt
\begin{lem}\label{shub:01}
Let \,$h\,,f:X \rightarrow \mathbb{F}$\, be continuous maps
where \,${\mathbb F}$\, is a Banach space and \,$X$\, is a metric space.
Suppose that \,$h$\, is injective, \,$h^{-1}$\, is Lipschitz and \,$f$\, satisfies
\,$\mathrm{Lip}(f-h)< \big[ \mathrm{Lip}(h^{-1})\big]^{-1}$. Then \,$f$\, is an
injective map such that
\begin{align}
 \mathrm{Lip}(f^{-1})  \leq \frac{\mathrm{Lip}(h^{-1})}{1-\mathrm{Lip}(f-h)\cdot \mathrm{Lip}(h^{-1})} \,.   &  \notag
\end{align}
\end{lem}

\vskip5pt


As a straightforward corollary we obtain the next result.


\vskip10pt
\begin{cor}\label{co:lip:inv}
Let \,$f:\RR\rightarrow\RR$\, be a map such that \,$\mathrm{Lip}(f-Id)<1$.
Then \,$f$\, is injective and
\begin{align}
 \mathrm{Lip}(f^{-1})  \leq \frac{1}{1-\mathrm{Lip}(f-Id)} \,.   &  \notag
\end{align}
\end{cor}
\vskip5pt


\vskip10pt
\begin{lem}\label{elondifeo}

If \,$f:\RR\rightarrow\RR$\, satisfies \,$\mathrm{Lip}(f-Id)\ <1$\,, then \,$f$\, is a homeomorphism.
\end{lem}

\begin{proof}
The map \,$f$\, is injective by Corollary \ref{co:lip:inv}. Thus the invariance of domain theorem
implies that \,$f({\mathbb R}^{2})$\, is an open set and that \,$f$\, is a homeomorphism onto its image.

Consider a sequence \,$(x_{n})_{n\geq 1}$\, in \,$\RR$\, such that \,$f(x_{n})\rightarrow p$\,.
The Lipschitz condition on \,$f$\, implies
\begin{align}
    \big\| \big(f(x_{n})-x_{n}\big) - \big(f(x_{m})-x_{m}\big) \big\|\leq
    \mathrm{Lip}(f-Id) \, \|x_{n}-x_{m} \|   &  \notag
\end{align}
and then
\begin{align}
    \big\| f(x_{n})-f(x_{m}) \big\|\geq \big(1-\mathrm{Lip}(f-Id)\big) \, \|x_{n}-x_{m}\| \,.
      &  \notag
\end{align}
We deduce that \,$(x_{n})_{n\geq1}$\, is a Cauchy sequence.
The limit \,$q$\, of \,$(x_{n})_{n\geq1}$\, satisfies \,$f(q)=p$\, since \,$f$\, is continuous.
Therefore \,$f(\RR)$\, is a closed set and \,$f:{\mathbb R}^{2} \to {\mathbb R}^{2}$\,
is a homeomorphism.
\end{proof}

\vskip10pt

Let us consider the homotopy \,$F_{t}(x)= t f(x)+(1-t)x$\, where \,$x\in\RR$\,
and \,$t\in[\,0\,,1\,]$.

\vskip10pt
\begin{cor}\label{isotopia}
Let \,$f:\RR\rightarrow\RR$\, be a map such that \,$\mathrm{Lip}(f-Id)<1$.
Then the homotopy \,$\big\{ F_{t}\big\}_{t\in[0\,,1]}$\,
is an isotopy relative to \,$\mathrm{Fix}(f)$\, of homeomorphisms of the plane.
\end{cor}

\begin{proof}
The map \,$F_{t}$\, satisfies
\,$\mathrm{Lip}(F_{t} - Id) = t \, \mathrm{Lip}(f-Id) <1$ for any $t \in [0,1]$.
The result is a consequence of Lemma \ref{elondifeo}.
\end{proof}

\vglue10pt

Given \,$\sigma \in \Z_{\geq 0}$\, we define\,:

\begin{itemize}

\item
$\epsilon_{0}:=\frac{1}{8}$ \quad  and \quad
$\displaystyle\epsilon_{\sigma} := \frac{1}{9\times6^{(\sigma-1)\sigma/2}}$\,
\ \ for $\sigma >0$\,;\\[-5pt]

\item
$\mathcal{V}_{\sigma}:=\big\{f \in \text{Homeo}(\RR)\; ;  \;
\mathrm{Lip}(f-Id) \leq\epsilon_{\sigma} \big\}$\,;  \\[-5pt]

\item
$\mathcal{U}:=\{f \in \text{Homeo}(\RR)\; ; \;
\mathrm{Lip}(f-Id) \leq \frac{1}{8} \big\}$.

\end{itemize}

\vskip10pt

\begin{remark}\label{obs-1}
The above definitions  imply\,:
\begin{itemize}
\item
$\mathcal{U} = \mathcal{V}_{0} \supset \mathcal{V}_{1}  \supset \cdots \supset \mathcal{V}_{\sigma}\supset \cdots$\,;
\end{itemize}
and
\begin{itemize}
\item
$f^{-1}\in\mathcal{U}$\, if \,$f\in\mathcal{V}_1$. It is a consequence of
item $(ii)$ of Lemma \ref{comutpropr} below.
\end{itemize}

\end{remark}


\vskip10pt
\begin{lem}\label{comutpropr}
Let \,$f,g:\RR \rightarrow \RR$\, such that \,$\mathrm{Lip}(f-Id)\leq a$\, and
\,$\mathrm{Lip}(g-Id)\leq b$. Then we obtain$\,:$
\begin{itemize}
\item[$(i)$]
$\displaystyle \mathrm{Lip}(f \circ g-Id)\leq a+b+ab \,;$
\item[$(ii)$]
$\displaystyle \mathrm{Lip}(f^{-1}-Id)\leq {a}/(1-a)$\, if \,$0\leq a<1$\,.
\end{itemize}
Moreover if \,$0\leq a\,,b\leq 1/9$\, we have$\,:$ \rule{0pt}{13pt}

\begin{itemize}
\item[$(iii)$]
$\displaystyle \mathrm{Lip}([\,f\,,g\,]-Id)\leq 6 \max\{a\,,b\} \,;$ \rule{0pt}{13pt}

\item[$(iv)$]
If \,$\sigma \in \Z^{+}$\, and \,$f_{1}\,,\ldots, f_{\sigma+1} \in \mathcal{V}_{\sigma+1}$\, then
\rule{0pt}{15pt}
$$[f_1\,,[f_2\,,\ldots,[f_i\,,f_{i+1}]\ldots]] \in \mathcal{V}_{\sigma} \quad
\text{for \ all} \quad i\in\{1\,,\ldots,\sigma\}\,.$$

\end{itemize}

\end{lem}
\vskip5pt


\begin{proof}
Let us show item $(i)$. Since
\[ f \circ g - Id = (f - Id) \circ g + (g - Id) \quad \text{and} \quad
\mathrm{Lip}(g) \leq \mathrm{Lip}(g - Id) + 1 \]%
we obtain
\,$\mathrm{Lip}(f \circ g - Id) \leq a (b+1) + b = a + b +ab$\,.

\vglue5pt

Suppose that \,$\mathrm{Lip}(f-Id)\leq a$\, with \,$0\leq a<1$.
The map \,$f$\, is a homeomorphism by Lemma \ref{elondifeo}. Moreover, we obtain
\[ \mathrm{Lip}(f^{-1} - Id) = \mathrm{Lip}((Id - f) \circ f^{-1}) \leq
a \cdot \mathrm{Lip}(f^{-1}) \leq \frac{a}{1-a}  \]
by Corollary \ref{co:lip:inv}.

\vglue5pt

Let us show item $(iii)$. We denote \,$A=\max\{a\,,b\}$. Since
\[ (f \circ g - g \circ f) = (f-Id) \circ g - (g - Id) \circ f + (g - Id) - (f - Id) \]
we obtain
\[ \mathrm{Lip}(f \circ g - g \circ f) \leq a (b+1) + b (a+1) + a + b = 2(ab +a +b)  \]
and
\[ \mathrm{Lip}([f,g] - Id) = \mathrm{Lip}[(f \circ g - g \circ f) \circ f^{-1} \circ g^{-1}]
\leq 2 \frac{ab + a +b}{(1-a)(1-b)} \leq 2 A \frac{A+2}{(1-A)^{2}} \]
by Corollary \ref{co:lip:inv}. It is straightforward to check out that the right hand side
of the above inequality is less or equal than \,$6 A$\, if \,$A \leq 1/9$.

Let \,$f_{1}\,,\ldots,f_{\sigma+1}\in \mathcal{V}_{\sigma+1}$\, and \,$1\leq i\leq \sigma$.
We apply the item $(iii)$ to obtain
$$\mathrm{Lip}([\,f_{i}\,,f_{i+1}\,]-Id)\leq
\frac{1}{9 \times 6^{(\sigma(\sigma+1)/2)-1}} \,.$$
Moreover, successive applications of item $(iii)$ allow to show
$$\mathrm{Lip}([f_1\,,[f_2\,,\ldots,[f_i\,,f_{i+1}]\ldots]]-Id)
\leq \frac{1}{9 \times 6^{(\sigma(\sigma+1)/2)-i}}\leq
\frac{1}{9 \times 6^{(\sigma-1)\sigma/2}}=\epsilon_{\sigma}$$
for any \,$1\leq i\leq \sigma$, completing the proof of item $(iv)$.
\end{proof}

\vskip10pt

We denote by \,$\mathrm{Cone}(q\,,v\,,\theta)$\, the cone of vertex
\,$q$\,, with axis equal to \,$\{ t v \ ; \ t \in {\mathbb R}_{\geq 0} \}$\,
and describing an angle of \,$2 \theta$\, radians where
\,$v \in {\mathbb S}^{1} \subset {\mathbb R}^{2}$, i.e.
$$\mathrm{Cone}(q\,,v\,,\theta):=\big\{q+u \ ; \  u\in\RR \ \ \mathrm{and} \ \
\text{Ang}(u\,,v)\leq \theta \big\}$$
where \,$0<\theta<\pi/2$\, and \,$\mathrm{Ang}(u\,,v)$\,
is by definition the angle comprised between the vectors \,$u$\, and \,$v$.

\vskip10pt

The next two lemmas are useful to control the image of segments \,$[\,p\,,q\,]$\, by mappings
\,$\epsilon$-Lipschitz with respect to the identity.


\vskip10pt
\begin{lem}\label{cc:01}
Let \,$0<\epsilon<1$ , $v\in\mathbb{S}^{1}$\, and \,$\mu \in {\mathbb R}^{+}$.
Consider a curve \,$\gamma:[\,0\,,\mu\,]\rightarrow\RR$\, such that
\,$\big\|\big(\gamma(t)-tv\big) - \big(\gamma(s)-sv \big)\big\|\leq \epsilon \, |t-s|$\,
for all \,$s\,,t\in[\,0\,,\mu\,]$\,. Then the curve \,$\gamma$\, is contained in
\,$\mathrm{Cone}\big(\gamma(0)\,,v\,,\arctan\big(\frac{\epsilon}{1-\epsilon}\big)\big)$\, and
it is injective. In particular, if \,$0 < \epsilon < 1/2$\, we obtain
\,$\gamma (\mu) \neq \gamma (0)$\, and \,$\gamma$\, is contained in
$$\mathrm{Cone}
\Big(\gamma(0)\,,\frac{\gamma(\mu)-\gamma(0)}{\|\gamma(\mu)-\gamma(0)\|}\,,2\arctan\big(\frac{\epsilon}{1-\epsilon}\big)\Big)\,.$$
\end{lem}


\begin{proof}
Up to an isometric change of coordinates we can suppose that \,$\gamma(0)$\, is the origin and
\,$v=(1\,,0)=e_{1}$. Thus the curve \,$\gamma(t)=\big(x(t)\,,y(t)\big)\in\RR$\, satisfies
\[ \big\| \big(\gamma(t)-te_{1}\big)-\big(\gamma(s)-se_{1}\big)\big\|\leq \epsilon \, |t-s|\,
\ \  \mathrm{in} \ \ [\,0\,,\mu\,]. \]
In particular we obtain \,$x(0)=y(0)=0$\,  and
\[
    \big| y(t) - y(s) \big| \leq \epsilon \, |t-s| \quad \mathrm{and} \quad
    \big| (x(t)-t) - (x(s) -s) \big| \leq \epsilon \, |t-s|\,
\ \ \mathrm{in} \  \ [\,0\,,\mu\,]. \]
The function \,$x(t)$\, is a non-negative function that is also
$\epsilon$-Lipschitz with respect to the identity. We obtain that \,$x(t)$\, is
injective by Lemma \ref{shub:01}.
Thus \,$\gamma$\, is an injective path.
Moreover, the inequality
\[ \frac{|y(t)|}{x(t)} \leq \frac{\epsilon t}{x(t)-t+t} \leq \frac{\epsilon t}{t(1-\epsilon)} =
\frac{\epsilon}{1-\epsilon} \]
implies the first statement in the lemma.

Since  \,$\gamma$\, is injective and \,$\gamma(0) \neq \gamma(\mu)$\,
then \,$\gamma$\, is contained in
$$\mathrm{Cone}\Big(\gamma(0)\,,\frac{\gamma(\mu)-\gamma(0)}{\|\gamma(\mu)-\gamma(0)\|}\,,2\arctan\big(\frac{\epsilon}{1-\epsilon}\big)\Big) \quad \text{if} \quad 0<\epsilon<1/2$$
completing the proof of the lemma.
\end{proof}


\vskip10pt
\begin{lem}\label{inj:los}
Let \,$f:\RR\rightarrow\RR$\, with \,$\mathrm{Lip}(f-Id)\leq \epsilon$.
Consider two different points \,$p\,,q\in\RR$. Then we have$\,:$
\begin{itemize}

\item[$(i)$]
for any \,$\lambda \in [\,p\,,q\,]$\, the point
\,$f(\lambda)$\, belongs to the intersection of the cones 
$$\mathrm{Cone}\Big(f(p)\,,\frac{w}{\|w\|}\,, 2\arctan\big(\frac{\epsilon}{1-\epsilon}\big)\Big)  \quad
\text{and} \quad \mathrm{Cone}\Big(f(q)\,,-\frac{w}{\|w\|}\,, 2\arctan\big(\frac{\epsilon}{1-\epsilon}\big)\Big)$$
if \,$0<\epsilon<\frac{1}{2}$\, where \,$w=f(q)-f(p) \,;$

\item[$(ii)$]
$\|f(\lambda)-f(p)\|\leq \|f(q)-f(p)\| \quad \text{for \ all} \quad \lambda \in [\,p\,,q\,]$\, \ \  if \rule{0pt}{13pt}
\ $0<\epsilon < \frac{1}{1+\sqrt{3}}$\,.   \rule{0pt}{17pt}

\end{itemize}
\end{lem}


\begin{proof}
We denote \,$v=\frac{q-p}{\|q-p\|}$\, and  \,$\gamma(t)=f(p+tv)$\, for
\,$t\in [\,0\,,\|q-p\|\,]$. We have
$$\gamma(0)=f(p) \quad \mathrm{and} \quad \gamma(\|q-p\|)=f(q) \,.$$
Since  \,$f$\, is $\epsilon$-Lipschitz with respect to the identity we obtain
 \begin{align}
    \big\| \big( \gamma(t)-tv \big) - \big( \gamma(s)-sv \big) \big\| & =
    \big\| \big( f(p+tv)-(p+tv)\big) -
     \big( f(p+sv)-(p+sv)\big) \big\|    \notag \\
    &  \leq \epsilon \, |t-s|        \notag
\end{align}
for all \,$s\,,t \in [\,0\,,\|q-p\|\,]$\,.

The points \,$f(p)$\, and \,$f(q)$\, are different if \,$0<\epsilon<1$\, by Lemma \ref{elondifeo}.
Moreover \,$\gamma$\, is contained in the cone
$$\mathrm{Cone}\Big(f(p)\,,\frac{w}{\|w\|}\,, 2\arctan\big(\frac{\epsilon}{1-\epsilon}\big)\Big)
\quad \mathrm{if} \quad 0<\epsilon < 1/2$$
by Lemma \ref{cc:01}.
We deduce that \,$\gamma$\, is contained in the cone
$$\mathrm{Cone}\Big(f(q)\,,-\frac{w}{\|w\|}\,, 2\arctan\big(\frac{\epsilon}{1-\epsilon}\big)\Big)
\quad \mathrm{if} \quad 0<\epsilon < 1/2$$
by interchanging the roles of \,$p$\, and \,$q$.

The proof is completed by noticing that when \,$0<\epsilon\leq1/(1+\sqrt{3})$\,
the intersection of the two cones defined above
is contained in the closed ball of center \,$f(p)$\, and radius \,$\|f(q)-f(p)\|$.
\end{proof}


\vskip30pt
\section{Flow-like properties of homeomorphisms $\epsilon$-Lipschitz with respect
to the identity}
\label{secabel}
\vglue5pt


Let \,$f$\, be a homeomorphism of \,$\RR$\, and \,$p \in \RR-\text{Fix}(f)$.
Consider an increasing sequence \,$(n_{k})_{k \geq 1}$\, of positive integer numbers.
Let us clarify that when we write \,$f^{n_{k}}(p) \to p$\, we are always assuming that
the sequence \,$f^{n_{k}}(p)$\, converges to \,$p$\, when \,$k$\, tends to \,$\infty$.

Given \,$m \geq 2$\, we denote by \,$\Gamma^f_{p,m}$\, the closed oriented curve obtained by
juxtaposing the line segments
$$[\,f(p)\,,f^2(p)\,]\,,[\,f^2(p)\,,f^3(p)\,]\,,\ldots,[\,f^{m-1}(p)\,,f^m(p)\,]\,,
[\,f^m(p)\,,f(p)\,]$$
where a segment \,$[\,f^{i}(p)\,,f^{i+1}(p)\,]$\, is oriented from
\,$f^{i}(p)$\, to \,$f^{i+1}(p)$.
The points
\,$f(p)$\,, $f^2(p)$\,, $\ldots,$ $f^m(p)$\, are the  \,{\it{vertices}}\, of \,$\Gamma^f_{p,m}$.
Analogously we define the curve \,$\Gamma^{f}_{\!\!p}$\, as the oriented curve obtained by
juxtaposing the segments \,$[\,f^{i}(p)\,,f^{i+1}(p)\,]$\, for \,$i\in\Z$\,.
We denote by \,$B(\,p\,,r\,)$\, (resp. \,$B[\,p\,,r\,]$\,) the open (resp. closed) ball of center
$p$ and radius \,$r>0$\,.

\vskip10pt

Roughly speaking the curves \,$\Gamma^{f}_{\!\!p}$\, can be interpreted
as the trajectories of a continuous dynamical system containing the orbits of \,$f$.
A significant issue is that even in very simple cases, for instance when
\,$f$\, belongs to the center of \,$G$, another element of the group does not preserve this superimposed
structure. More precisely we have \,$h(\Gamma_{p}^{f}) \neq \Gamma_{h(p)}^{f}$\,
for general \,$p \in {\mathbb R}^{2} - \mathrm{Fix}(f)$\, and \,$h \in G$.
This does not constitute a problem, since the continuous dynamical system
is preserved up to homotopy relative to \,$\mathrm{Fix}(f)$.

\vskip10pt

\begin{lem}\label{homot}
Let \,$G \subset \mathrm{Homeo}(\mathbb{R}^2)$\, be a group.
Consider a normal subgroup \,$L$\, of \,$G$\, and \,$f \in {\mathcal U} \cap G$\, such that
the class of \,$f$\, in \,$G/L$\, belongs to \,$Z(G/L)$.
Let \,$p \in \mathrm{Fix}(L) - \mathrm{Fix}(f)$\, and \,$h \in \mathcal{U}\cap G$.
Suppose that there exists a sequence \,$(n_{k})_{k \geq 1}$\, of positive integers
such that \,$f^{n_{k}}(p) \to p$.
Then there exists a homotopy in \,${\mathbb R}^{2} - \mathrm{Fix}(f)$\, relative to vertices
between the curves
\,$\Gamma_{h(p),n_{k}}^{f}$\, and \,$h(\Gamma_{\!\!p,n_{k}}^{f})$\, for \,$k$\, big enough.
\end{lem}

\begin{proof}
The commutator \,$[\,h^{-1},f^{-j}\,]$\, belongs to \,$L$\, for any
\,$j \in {\mathbb Z}$\,
since \,$f \in Z(G/L)$.
The point \,$p$\, is a global fixed point of \,$L$, hence we have
$$h(f^j(p))=f^j(h(p)) \quad \mathrm{for \ all} \quad j \in \mathbb{Z} \,.$$
Given \,$j \in {\mathbb Z}$\, we apply Lemma \ref{inj:los} to the map \,$f$\, and
the segment \,$[\,f^j(p)\,,f^{j+1}(p)\,]$.
We obtain
$$\|h(\lambda)-f^j(h(p))\|=\|h(\lambda)-h(f^j(p))\| \leq \|h(f^{j+1}(p))-h(f^j(p))\|$$
$$\leq \|f^{j+1}(h(p))-f^j(h(p))\|$$
for any \,$\lambda \in [\,f^j(p)\,,f^{j+1}(p)\,]$.
In particular \,$h [\,f^j(p)\,,f^{j+1}(p)\,]$\, is contained in the closed ball \,$B_{j}$\, of center
\,$f^{j}(h(p))$\,
and radius \,$\|f^{j+1}(h(p))-f^j(h(p))\| \,$.
Moreover, corollary \ref{co:dpi} implies that \,$B_{j}$\, does not contain fixed points of \,$f$.
The curves
\[ [\,f^{j}(h(p))\,,f^{j+1}(h(p))\,]\, \ \mathrm{and} \ \,h\big([\,f^j(p)\,,f^{j+1}(p)\,]\big) \]
are contained in \,$B_{j}$\,, therefore they are homotopic via an homotopy relative
to ends in \,$B_{j}$.

\vglue10pt

{\hfil
\setlength{\unitlength}{1mm}
\begin{picture}(132,42)(-1,20)
\put(0,0){\includegraphics[scale=0.40]{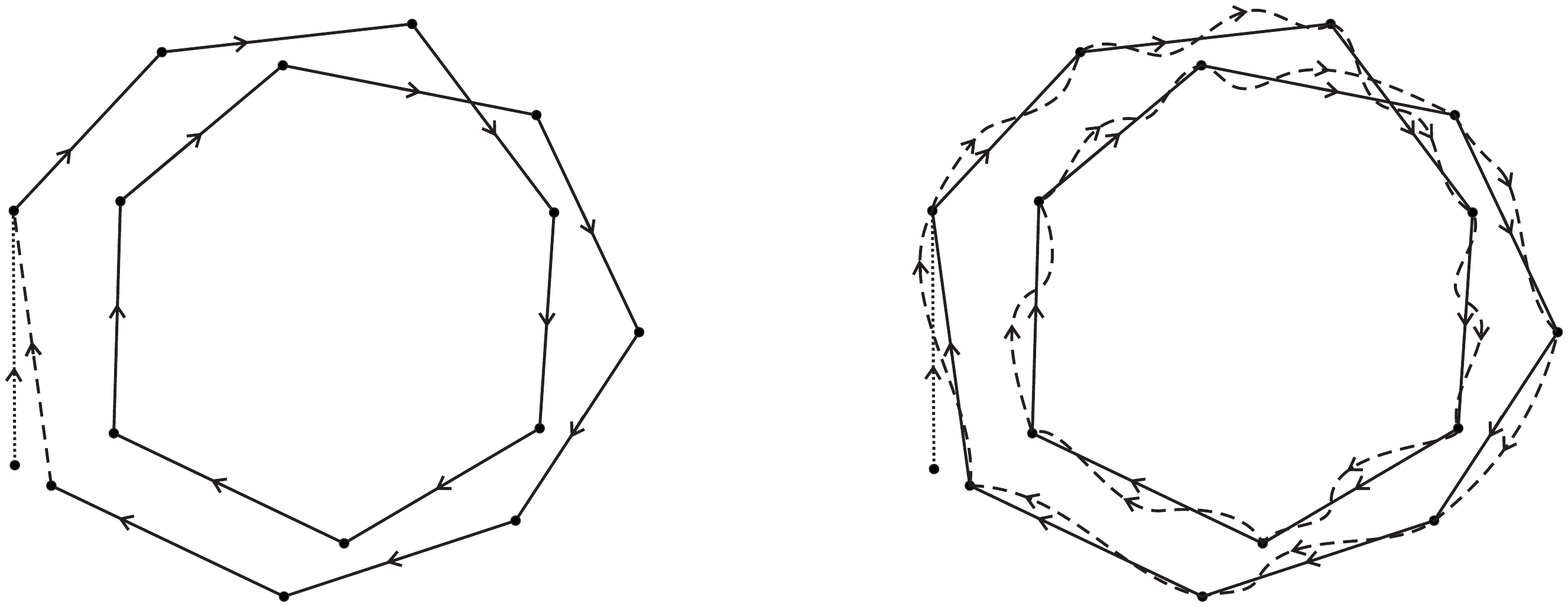}}%
\put(4,30){\footnotesize$p$}
\put(63,29){\footnotesize$h(p)$}
\put(-1,47){\footnotesize$f(p)$}
\put(57,47){\footnotesize$f(h(p))$}
\put(7,59){\footnotesize$f^2(p)$}
\put(66,59){\footnotesize$f^2(h(p))$}
\put(44,54){\footnotesize$f^j(p)$}

\put(107,56){\footnotesize$\swarrow$}
\put(110,57){\footnotesize$f^j(h(p))$}

\put(114,41){\footnotesize$\swarrow$}
\put(114,44){\footnotesize$f^{j+1}(h(p))$}

\put(50,38){\footnotesize$f^{j+1}(p)$}
\put(5,22){\footnotesize$f^{n_k}(p)$}
\put(5.5,25.5){\footnotesize$\nearrow$}
\put(45,28){$\Gamma^f_{\!\!p,n_k}$}
\put(111,32){$\leftarrow h\big(\Gamma^f_{\!\!p,n_k}\big)$}

\put(67,21){\footnotesize$f^{n_k}(h(p))$}
\put(68.5,25.5){\footnotesize$\nearrow$}

\put(82,49){$\nwarrow$}
\put(87,47){$\Gamma^f_{\!\!h(p),n_k}$}

\end{picture}}

\vglue10pt

It remains to show that
\begin{equation}
\label{equ:curves}
[\,f^{n_k}(h(p))\,,f(h(p))\,] \;\; \mathrm{and} \;\; h \big([\,f^{n_k}(p)\,,f(p)\,]\big)
\end{equation}
are also homotopic relative to ends in \,${\mathbb R}^{2} - \mathrm{Fix}(f)$.
By the first part of the proof we have that
\,$[\,h(p)\,,f(h(p))\,]$\, and \,$h([\,p\,,f(p)\,])$\, are contained in \,$B_{0}$\,.
Since \,$f^{n_{k}}(p) \to p$\, and \,$f^{n_{k}}(h(p)) = h(f^{n_{k}}(p)) \to h(p)$\,
we deduce that the curves in expression (\ref{equ:curves}) are contained in
the closed ball of center \,$h(p)$\, and radius \,$2 \|f(h(p))-h(p) \| \,$
for \,$k >>0$.
We argue as above since such a ball does not contain points of \,$\mathrm{Fix}(f)$\,
by Corollary \ref{co:dpi}.

\end{proof}

Some properties of \,$\Gamma^{f}_{\!\!p}$\, and \,${\mathcal O}_{p}(f)$\, are analogous.
For instance next lemma
applied to \,$K=\overline{{\mathcal O}_{p}(f)}$\,
implies that a fixed point of \,$f$\, is in the closure of the $f$-orbit
of \,$p$\, if and only if it belongs to the closure of  \,$\Gamma^{f}_{\!\!p}$.
\begin{lem}
\label{lem:iso}
Let \,$f \in {\mathcal U}$. Consider a compact set \,$K$\, and a point
\,$q \in \mathrm{Fix}(f) - K$.
Then there exists \,$\delta >0$\, such that \,$B(q\,,\delta) \cap [\,y\,,f(y)\,]=\emptyset$\,
for any \,$y \in K$.
\end{lem}

\begin{proof}
There exists \,$\delta>0$\, such that \,$B(q,2 \delta) \cap K = \emptyset$.
We claim that there is no point \,$z$\, in \,$B(q,\delta) \cap [y,f(y)]$\,
for any \,$y \in K$.
Otherwise \,$y \not \in \mathrm{Fix}(f)$\, and
the length of \,$[\,y\,,f(y)\,]$\, is equal to the sum of the lengths
of \,$[\,y\,,z\,]$\, and \,$[\,z\,,f(y)\,]$\, and as a consequence greater than \,$\delta$.
Since
\[ ||y-q|| \leq ||y-z|| + ||z-q|| < ||y-f(y)|| +\delta  < 2 ||y-f(y)|| \]
we obtain a contradiction with Corollary \ref{co:dpi}.
\end{proof}

Next we define a concept of fixed point of a homeomorphism
\,$f$\, of \,$\RR$\, enclosed by an orbit of \,$f$.


\begin{defn*}
Let \,$f$\, be a homeomorphism of \,$\RR$\, and \,$p \in \RR-\mathrm{Fix}(f)$.
We say that a point \,$q \in \mathrm{Fix}(f)$\,
is a \,{\it capital point}\, for \,$\mathcal{O}_p(f)$\, if there exists
an increasing sequence of positive integers \,$(n_{k})_{k\geq1}$\, such that\,:
\begin{itemize}
\item $f^{n_k}(p)\to p\,;$
\item $\mathrm{Ind}_q(\Gamma^f_{\!\!p,n_k})$\, is a well-defined non-vanishing integer number for
\,$k >>0$.

\end{itemize}
\end{defn*}


\vglue5pt

In the above definition  \,$\mathrm{Ind}_q(\Gamma^f_{\!\!p,n_k})$\, stands for the
\,\emph{winding number}\, of the curve \,$\Gamma^f_{\!\!p,n_k}$\, with respect to the
point \,$q$.

{\hfil
\setlength{\unitlength}{1mm}
\begin{picture}(60,48)(12,70)
\put(0,0){\includegraphics[scale=0.40]{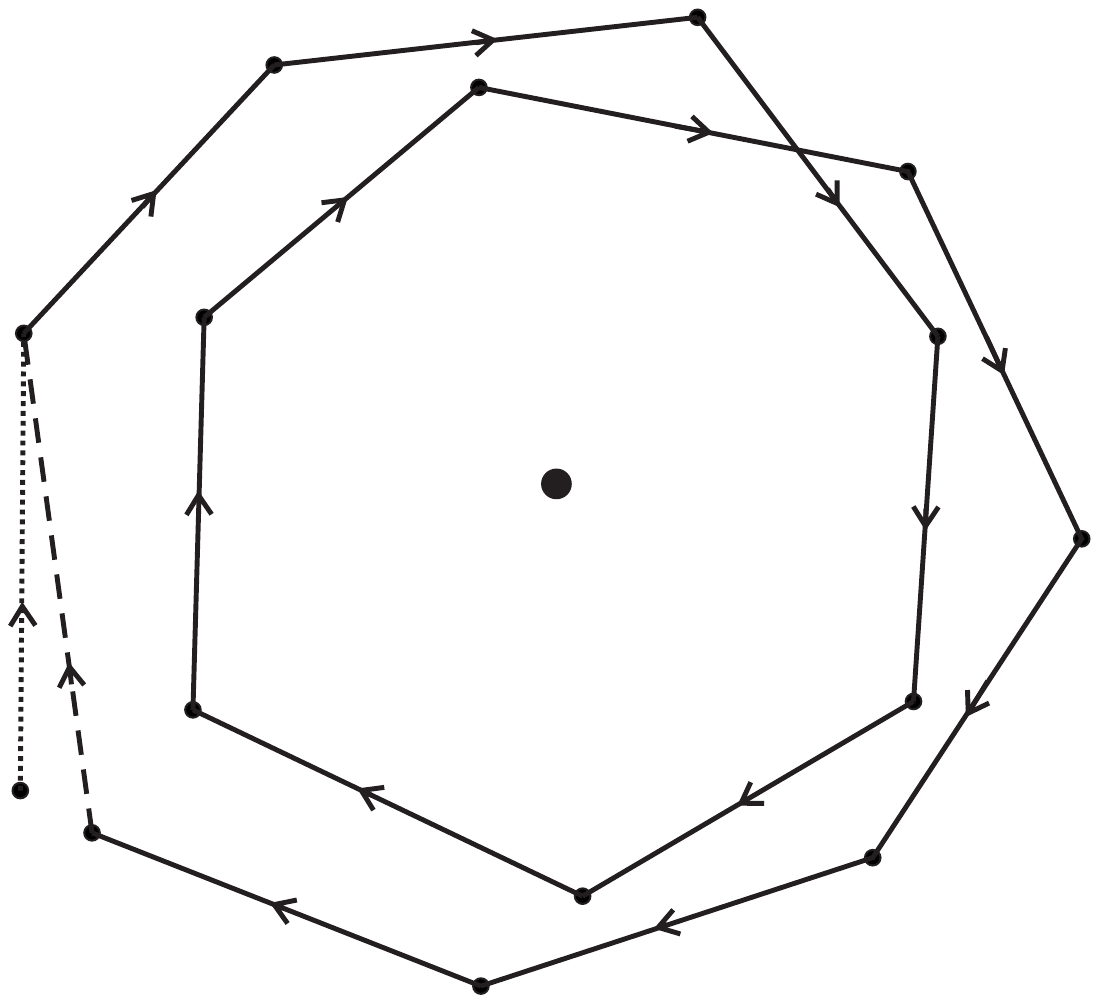}}%
\put(40,90){\footnotesize$q \in \mathrm{Fix}(f)$}
\put(18.5,80){\footnotesize$p$}
\put(14,99){\footnotesize$f(p)$}
\put(22,110){\footnotesize$f^2(p)$}
\put(50,113){\footnotesize$f^3(p)$}
\put(20,76){\footnotesize$\nearrow$}
\put(17,72){\footnotesize$f^{n_k}(p)$}
\put(63,96){$\Gamma^f_{\!\!p,n_k}$}
\end{picture}}

\vglue10pt

Lemmas \ref{ind:pt:02} and \ref{fechopontocap} show that the set of capital points
is closed and invariant by the action of subgroups under suitable conditions.
The next remark will be useful in the proofs.

\vskip10pt

\begin{remark}\label{inv:fix}
Let \,$G \subset \mathrm{Homeo}(\mathbb{R}^2)$\, be a group.
Consider a normal subgroup \,$L$\, of \,$G$\, and \,$f \in G$\, such that
the class of \,$f$\, in \,$G/L$\, belongs to \,$Z(G/L)$.
Then \,$\langle L\,,f \rangle$\, is a normal subgroup of \,$G$\,
and the set \,$\mathrm{Fix}(L\,,f)$\, is $G$-invariant.
\end{remark}

\vglue5pt

\begin{lem}\label{ind:pt:02}
Let \,$G \subset \mathrm{Homeo}(\mathbb{R}^2)$\, be a group.
Consider a normal subgroup \,$L$\, of \,$G$\, and \,$f \in {\mathcal U} \cap G$\, such that
the class of \,$f$\, in \,$G/L$\, belongs to \,$Z(G/L)$.
Let \,$p \in \mathrm{Fix}(L) - \mathrm{Fix}(f)$\, and \,$h \in \mathcal{V}_{1} \cap G$.
Consider a capital point \,$q \in \mathrm{Fix}(L\,,f)$\, for \,$\mathcal{O}_p(f)$.
Then \,$h^{\ell}(q) \in \mathrm{Fix}(L\,,f)$\, is a capital point for
\,$\mathcal{O}_{h^{\ell}(p)}(f)$\, for any \,$\ell \in \mathbb{Z}$\,.
\end{lem}

\begin{proof}
The point \,$h(q)$\, belongs to \,$\mathrm{Fix}(L\,,f)$\, by Remark \ref{inv:fix}.
Moreover, since \,$q$\, is a capital point for \,$\mathcal{O}_p(f)$\,
there exists an increasing sequence \,$(n_{k})_{k \geq 1}$\, of positive integers such that\,:
\begin{itemize}
\item $f^{n_k}(p) \rightarrow p\,;$
\item $\mathrm{Ind}_{q}(\Gamma^f_{\!\!p,n_{k}})$\, is a well-defined non-vanishing integer number for
\,$k>>0$.
\end{itemize}
Since the class of \,$f$\, belongs to \,$Z(G/L)$\, and \,$p\in \mathrm{Fix}(L)$\, we have
\,$h\big(f^j(p)\big)=f^j\big(h(p)\big)$ for any $j \in \mathbb{Z}$.
In particular we obtain
\,$f^{n_k}\big(h(p)\big)\rightarrow h(p)$\, and \,$h(p)\notin \text{Fix}(f)$.
Since \,$h$\, is orientation-preserving
the winding number \,$\text{Ind}_{h(q)}\big(h(\Gamma^f_{\!\!p,n_k})\big)$\, is well-defined,
equal to \,$\text{Ind}_q(\Gamma^f_{\!\!p,n_k})$\, and then non-vanishing for \,$k>>0$.

The curves \,$h(\Gamma^f_{\!\!p,n_k})$\, and \,$\Gamma^f_{\!\!h(p),n_k}$\, are homotopic relative to vertices
via an homotopy in \,$\RR - \mathrm{Fix}(f)$\, by Lemma \ref{homot}.
Thus we obtain
$$\text{Ind}_{h(q)}(\Gamma^f_{\!\!h(p),n_k}) = \text{Ind}_{h(q)}\big(h(\Gamma^f_{\!\!p,n_k})\big) \neq 0
\quad \mathrm{for} \quad k >>0 \,.$$
We deduce that
\,$h(q)$\, is a capital point for \,${\mathcal O}_{h(p)}(f)$.
Hence \,$h^{\ell}(q)$\, is a capital point for \,${\mathcal O}_{h^{\ell}(p)}(f)$\, for any \,$l \geq 0$\,
by successive applications of the previous argument.
We just used that \,$h$\, belongs to \,${\mathcal U}$.
Since \,$h^{-1} \in {\mathcal U}$\, whenever \,$h \in {\mathcal V}_{1}$\, we obtain  that
\,$h^{\ell}(q)$\, is a capital point for \,${\mathcal O}_{h^{\ell}(p)}(f)$\, for any
\,$l \in {\mathbb Z}$.
\end{proof}


\vskip10pt
\begin{lem} \label{fechopontocap}
Let \,$G \subset \mathrm{Homeo}({\mathbb R}^{2})$\, be a group.
Let \,$L$\, be a normal subgroup of \,$G$.
Let \,$h_{1},\hdots,h_{n} \in G \cap {\mathcal V}_{1}$\,
and \,$f \in G \cap {\mathcal U}$\, such that
the class of \,$f$\, in \,$G/L$\, belongs to \,$Z(G/L)$.
Let \,$q \in \mathrm{Fix}(L,f)$\, be a capital point for \,${\mathcal O}_{p}(f)$\,,
where \,$p \in \mathrm{Fix}(L) - \mathrm{Fix}(f)$\, has bounded
\,$(h_{1},\hdots,h_{n},f)$-orbit. If
\[ q' \in \overline{{\mathcal O}_{q}(h_{1},\hdots,h_{n})} \ \ \text{and} \ \
q' \not \in \overline{{\mathcal O}_{p}(h_{1},\hdots,h_{n},f)} \]
then \,$q'$\, is a capital point for \,${\mathcal O}_{z}(f)$\, for some
\,$z \in {\mathcal O}_{p}(h_{1},\hdots,h_{n})$.
\end{lem}


\begin{proof}
We denote \,$K=\overline{{\mathcal O}_{p}(h_{1},\hdots,h_{n},f)}$.
Lemma \ref{lem:iso} and Remark \ref{inv:fix} imply the existence of
\,$\delta >0$\, such that \,$B(q',\delta) \cap [\,y\,,f(y)\,]=\emptyset$\,
for any \,$y \in K$. On the other hand
there exists a sequence \,$(\varphi_{\ell})_{\ell \geq 1}$\, of elements of the
group \,$\langle h_{1},\hdots,h_{n} \rangle$\, such that \,$\varphi_{\ell}(q)\to q'$.
We fix \,$\ell \in {\mathbb N}$\, such that \,$\|q' - \varphi_{\ell}(q)\| < \delta /2$.
Since \,$q$\, is a capital point for \,${\mathcal O}_{p}(f)$\, we obtain that
\,$\varphi_{\ell}(q)$\, is a capital point for \,${\mathcal O}_{\varphi_{\ell}(p)}(f)$\,
by Lemma \ref{ind:pt:02}.
Let \,$(n_{k})_{k \geq 1}$\, be an increasing sequence such that
\,$f^{n_{k}}(\varphi_{\ell}(p)) \to \varphi_{\ell}(p)$\, and
\,$\mathrm{Ind}_{\varphi_{\ell}(q)}(\Gamma^f_{\!\!\varphi_{\ell}(p),n_{k}})$\, is well-defined
and non-vanishing for any \,$k \in {\mathbb N}$.
The curve \,$\Gamma^f_{\!\!\varphi_{\ell}(p),n_{k}}$\, does not intersect \,$B(q',\delta/2)$\, for \,$k>>0$\,
by construction.
As a consequence \,$q'$\, is a capital point for \,${\mathcal O}_{\varphi_{\ell}(p)}(f)$.
\end{proof}

\vskip10pt

Consider homeomorphisms \,$f\,,g$\, that are embedded in  topological flows,
i.e. \,$f = \mathrm{exp}(X)$\, and \,$g= \mathrm{exp}(Y)$.
Suppose \,$p \in \mathrm{Sing}(Y) - \mathrm{Sing}(X)$\,
and \,$q \in \mathrm{Sing}(X) - \mathrm{Sing}(Y)$.
It is obvious that if the flows commute then the trajectory of \,$X$\, through \,$p$\,
is contained in \,$\mathrm{Sing}(Y) - \mathrm{Sing}(X)$\,
and the trajectory of \,$Y$\, through \,$q$\,
is contained in \,$\mathrm{Sing}(X) - \mathrm{Sing}(Y)$.
In particular the trajectories do not intersect. The next lemma
is the generalization of this flow behavior in the $\epsilon$-Lipschitz
with respect to the identity setting.


\vskip10pt
\begin{lem}\label{k:gr:kapa}
Let \,$G \subset \mathrm{Homeo}(\mathbb{R}^2)$\, be a subgroup and  \,$f,g\in \mathcal{U}\cap G$.
Suppose that
$$p \in \mathrm{Fix}(G_{(1)}, g)-\mathrm{Fix}(f) \quad \mathrm{and} \quad
q \in \mathrm{Fix}(G_{(1)}, f)-\mathrm{Fix}(g)\,.$$
Then the following properties are satisfied$\,:$
\begin{itemize}
\item[$(i)$]
The curves \,$\Gamma^{f}_{\!\!p}$\, and \,$\Gamma^{g}_{\!\!q}$\, are disjoint$\,;$

\item[$(ii)$]
If there exists a constant \,$r>0$\, such that
$$\|f^{i+1}(p)-f^i(p)\|\geq r \quad \mathrm{and} \quad \|g^{i+1}(q)-g^{i}(q)\|\geq r$$
for any \,$i\in\Z$\,, then \,$d(\Gamma^f_{\!\!p},\Gamma^g_{\!\!q})\geq r$.

\item[$(iii)$]
If there is an increasing sequence \,$(n_k)_{k\geq 1}$\, of positive integers such that
\,$f^{n_k}(p) \rightarrow p$\, then there exists \,$\kappa\in\Z^{+}$\, such that
\,$\Gamma^f_{\!\!p,n_k} \cap \Gamma^g_{\!\!q} =\emptyset$\, for all \,$k\geq \kappa$\,
and \,$\Gamma^g_{\!\!q}$\,.
\end{itemize}

\end{lem}
\vskip5pt


\begin{proof}
We will proof only item $(iii)$ since the proofs of items $(i)$ and $(ii)$
use analogous arguments.
Moreover,
the items $(i)$ and $(ii)$ are versions of Lemmas 4.3 and 4.5 of \cite{dff02} respectively.
The proofs are essentially the same, using Remark \ref{inv:fix} and item $(iii)$
of Corollary \ref{co:dpi}.

Suppose that item $(iii)$ does not hold true. Up to consider a subsequence we
can suppose that  \,$\Gamma^f_{\!\!p,n_k} \cap \Gamma^g_{\!\!q}  \neq \emptyset$\,
for any \,$k \in {\mathbb N}$. Since
\,$\Gamma^f_{\!\!p} \cap \Gamma^g_{\!\!q} = \emptyset$\, by item $(i)$ we obtain
$$[\,f^{n_k}(p)\,, f(p)\,] \cap \Gamma^g_{\!\!q} \neq \emptyset \quad  \textrm{for any} \quad
k \in {\mathbb N}\,.$$
Let \,$j \in {\mathbb Z}$\, such that
\,$[\,f^{n_k}(p)\,, f(p)\,] \cap [\,g^j(q)\,,g^{j+1}(q)\,] \neq \emptyset .$
Of course \,$j$\, depends on \,$k$.
We obtain
$$\|g^j(q)-f^{n_k}(p)\| \leq \|g^{j+1}(q)-g^{j}(q)\|+\|f^{n_k}(p)-f(p)\|\leq$$
$$\leq 2 \max\{\|g^{j+1}(q)-g^j(q)\|\,,\|f^{n_k}(p)-f(p)\|\}.$$
There exist two cases:
\begin{itemize}
\item
$f^{n_k}(p)$\, belongs to the closed ball of center \,$g^j(q)$\,
and radius \,$2\|g^{j+1}(q)-g^j(q)\|$.
This is impossible since \,$f^{n_k}(p)$\, is a fixed point of \,$g$\, by Remark \ref{inv:fix}
and the ball does not contain fixed points of \,$g$\, by Corollary \ref{co:dpi}\,;

\item
$g^j(q)$\, belongs to the closed ball of center \,$f^{n_k}(p)$\,
and radius \,$2\|f^{n_k}(p)-f(p)\|$. Since \,$f^{n_{k}}(p) \to p$\,
the point \,$g^{j}(q)$\, is a fixed point of \,$f$\, (Remark \ref{inv:fix})
that belongs to the closed ball of center \,$p$\, and radius
\,$4 \|p-f(p)\|$\, for \,$k >>0$. This contradicts Corollary \ref{co:dpi}.
\end{itemize}
We obtain a contradiction in both cases.
\end{proof}


\vskip30pt
\section{Proof of  Theorem \ref{teorema}}
\vskip10pt

Our first goal is making explicit that a finitely generated nilpotent
group is a tower of cyclic central extensions.
 Let ${\mathcal S}$ be a subset of a group $G$. We define
\[ {\mathcal S}_{(0)}:= {\mathcal S} \quad \mathrm{and} \quad
{\mathcal S}_{(j+1)}:= \{ [\,a\,,b\,] \ ; \ a \in {\mathcal S} \ \mathrm{and} \
b \in  {\mathcal S}_{(j)} \} \quad \mathrm{for} \quad j \geq 0 \,.\]

Suppose that \,${\mathcal S}$\, generates a nilpotent group \,$G$.
The next result implies that we can find a generator set of
any subgroup in the descending central series by
considering iterated commutators of elements in \,${\mathcal S}$.
It is a generalization of Proposition 2.3 of \cite{dff02}.

\begin{lem}
\label{lem:censer}
Let \,$G$\, be a $\sigma$-step nilpotent group generated by a subset \,${\mathcal S}$\, of \,$G$.
Then we obtain \,$G_{(j)}=\langle {\mathcal S}_{(j)}, \hdots, {\mathcal S}_{(\sigma -1)} \rangle$\,
for all \,$\sigma \geq 1$\, and \,$j \geq 0$.
\end{lem}

\begin{proof}
Any $\sigma$-step nilpotent group generated by \,${\mathcal S}$\, satisfies
\,$G_{(\sigma -1)} =\langle {\mathcal S}_{(\sigma -1)} \rangle$\, by Lemma 2.2 of \cite{dff02}.
This result implies the statement of the lemma for \,$\sigma \leq 2$.

Let us show that if the lemma holds true for \,$\sigma$\, then so it does for
\,$\sigma +1$. The series
\[ \frac{G}{G_{(\sigma)}} \varsupsetneq \frac{G_{(1)}}{G_{(\sigma)}} \varsupsetneq
\cdots \varsupsetneq \frac{G_{(\sigma -1)}}{G_{(\sigma)}} \varsupsetneq
\frac{G_{(\sigma)}}{G_{(\sigma)}} = \{ \overline{Id} \} \]
is the lower central series of the $\sigma$-step nilpotent group
\,$G/G_{(\sigma)}$\,
since \,$(G/G_{(\sigma)})_{(j)} = G_{(j)} /G_{(\sigma)}$\, for any \,$0 \leq j \leq \sigma$.
We obtain
\[ \frac{G_{(j)}}{G_{(\sigma)}} = \langle \overline{\mathcal S}_{(j)}, \hdots, \overline{\mathcal S}_{(\sigma -1)} \rangle \]
by induction hypothesis where \,$\overline{Id}$\, and \,$\overline{\mathcal S}_{(k)}$\, are the projections of \,${Id}$\, and  \,${\mathcal S}_{(k)}$\,  in \,$G/G_{(\sigma)}$\,
respectively.
Therefore \,$G_{(j)}$\, is equal to
\,$\langle {\mathcal S}_{(j)}, \hdots, {\mathcal S}_{(\sigma -1)},G_{(\sigma)} \rangle$\, for any
\,$j \geq 0$.
We obtain
\,$G_{(j)} =\langle {\mathcal S}_{(j)}, \hdots, {\mathcal S}_{(\sigma -1)},{\mathcal S}_{(\sigma)} \rangle$\, for any \,$j \geq 0$\,
by Lemma 2.2 of \cite{dff02}.
\end{proof}

\vskip10pt

\begin{remark}
\label{rem:loss}
Let \,$\sigma \geq 1$.
A direct consequence of the previous lemma and item $(iv)$
of Lemma \ref{comutpropr}  is that if a
\,$(\sigma +1)$-step nilpotent subgroup \,$G$\, of \,$\mathrm{Homeo}({\mathbb R}^{2})$\, is generated by a subset
of \,${\mathcal V}_{\sigma+1}$\, then \,$G_{(j)}$\, is generated by a subset of
\,${\mathcal V}_{\sigma}$\, for any \,$j \geq 0$.
\end{remark}

The next lemma is a version of the Main Lemma of \cite[p. 1085]{dff02}
for homeomorphisms of the plane that are $\epsilon$-Lipschitz with respect
to the identity. It is the key tool in order to find global fixed points.

\begin{lem}
\label{lem:fdertog}
Let \,$G \subset \mathrm{Homeo}({\mathbb R}^{2})$\, be a $\sigma$-step nilpotent group
finitely generated in \,${\mathcal V}_{\sigma}$.
Let \,$g_{0}\,,\ldots,g_{m}\,,f \in {\mathcal V}_{\sigma} \cap G$\,
and \,$p \in \mathrm{Fix} (G_{(1)}, g_{0},\ldots, g_{m}) - \mathrm{Fix}(f)$\,
where \,$m \geq 0$.
Suppose there exists an increasing sequence \,$(n_{k})_{k \geq 1}$\, of positive integers
such that \,$f^{n_{k}}(p) \to p$.
If \,$\gamma_{k}$\, is a simple closed curve contained in \,$\Gamma_{\!\!p,n_{k}}^{f}$\, then
there exists
\[ q_{k} \in \mathrm{Fix} (G_{(1)}, g_{0},\hdots, g_{m},f) \cap \mathrm{Int}({\mathcal D}_{k}) \]
for \,$k$\, big enough where \,${\mathcal D}_{k}$\, is the compact disc bounded by
\,$\gamma_{k}$.
\end{lem}

The map \,$g_{0}$\, is always the identity map by convention.

\begin{proof}
The proof is by induction on \,$\sigma$.
It is convenient to consider that the case \,$\sigma=0$\, of
the induction process corresponds to the situation
\,$G=\langle f \rangle$\, and \,$g_{0} \equiv  \cdots \equiv g_{m} \equiv Id$.
In this way we avoid a special proof for the
case \,$\sigma=1$.
The case \,$\sigma=0$ is a consequence of Lemma \ref{lem:ind:2}.

Let us show that if the lemma holds true for any group of
nilpotent class \,$l$\, with \,$0 \leq l \leq \sigma$\, then
it holds true for any group \,$G$\, of nilpotent class \,$\sigma + 1$.
We consider a second induction process on \,$m \geq 0$.
The case \,$m=0$\, is simple since
\,$H:=\langle G_{(1)}, f \rangle$\, is a nilpotent group whose nilpotency class is less or equal
than \,$\sigma$\, if \,$\sigma >0$\,; indeed we have \,$H_{(1)} \subset G_{(2)}$.
We obtain \,$H=\langle f \rangle$\, for the case \,$\sigma=0$, this is the reason because we choose
a modified induction hypothesis.
Let \,$\{h_{1},\ldots,h_{n} \} \subset {\mathcal V}_{\sigma}$\, be a generator set
of \,$G_{(1)}$\, provided by Remark \ref{rem:loss}. Since
\[ p \in \mathrm{Fix} ( H_{(1)}, h_{1}, \hdots ,h_{n} ) - \ \mathrm{Fix}(f) \]
then there exists
\,$q_{k} \in \mathrm{Fix}(H) \cap \mathrm{Int}({\mathcal D}_{k})$\, by induction
hypothesis for \,$k >>0$.

Suppose that the result holds true for some \,$m \geq 0$.
We denote
\[ A=\langle G_{(1)}, g_{0},\hdots, g_{m+1} \rangle \quad \mathrm{and} \quad B=\langle G_{(1)}, g_{0},\ldots, g_{m},f \rangle. \]
Let
\,$p \in \mathrm{Fix}(A) - \mathrm{Fix}(f)$.
If the result in the lemma is not satisfied we can consider
\,$\mathrm{Fix} ( A\,,f) \cap \mathrm{Int}({\mathcal D}_{k}) = \emptyset$\,
for any \,$k \in {\mathbb N}$\, up to consider
a subsequence of \,$(n_{k})_{k \geq 1}$.
Fix \,$k \in {\mathbb N}$\, big enough.
The induction hypothesis implies that there exists
\,$y_{0} \in \mathrm{Fix}(B)  \cap \mathrm{Int}({\mathcal D}_{k})$\,
for any \,$k >>0$.
We apply the item $(iii)$ of Lemma \ref{k:gr:kapa} to the diffeomorphisms \,$f$\, and \,$g_{m+1}$\,
and the points \,$p$\, and \,$y_{0}$.
We obtain \,$\Gamma_{\!\!p,n_{k}}^{f} \cap \Gamma_{\!\!y_{0}}^{g_{m+1}} = \emptyset$\, for \,$k>>0$.
In particular the curves \,$\gamma_{k}$\, and \,$\Gamma_{\!\!y_{0}}^{g_{m+1}}$\, are disjoint
and \,$\Gamma_{\!\!y_{0}}^{g_{m+1}}$\, is contained in \,$\mathrm{Int}({\mathcal D}_{k})$.
The orbit \,${\mathcal O}_{y_{0}}(g_{m+1})$\, is contained in \,$\Gamma_{\!\!y_{0}}^{g_{m+1}}$\,
and then in \,$\mathrm{Int}({\mathcal D}_{k})$.
Since \,$f$\, does not have fixed points in \,$\gamma_{k}$\, by
Corollary \ref{co:dpi}
and \,$\overline{{\mathcal O}_{y_{0}}(g_{m+1})}$\, is contained in
\,$\mathrm{Fix}(f)$\, we deduce that
\,$\overline{{\mathcal O}_{y_{0}}(g_{m+1})}$\, is contained in \,$\mathrm{Int}({\mathcal D}_{k})$.
There exist $\omega$-recurrent points for \,$g_{m+1}$\, in
\,$\overline{{\mathcal O}_{y_{0}}(g_{m+1})}$.
Hence we can suppose that \,$y_{0}$\, is a $\omega$-recurrent point for \,$g_{m+1}$\, by
replacing \,$y_{0}$\, if necessary with another point in \,$\overline{{\mathcal O}_{y_{0}}(g_{m+1})}$.
There exists a simple closed curve \,$\alpha_{0}$\, contained in
\,$\Gamma_{\!\!y_{0}}^{g_{m+1}}$\, by Lemma \ref{le:cs}.
The compact disc \,$\Delta_{0}$\, bounded by \,$\alpha_{0}$\, is contained in
\,$\mathrm{Int}({\mathcal D}_{k})$.

Let \,$(l_{k})_{k \geq 1}$\, be an increasing sequence of positive integers
such that \,$g_{m+1}^{l_{k}}(y_{0}) \to y_{0}$.
Up to replace \,$y_{0}$\, with another point in the \,$g_{m+1}$-orbit
of \,$y_{0}$\, we can suppose \,$\alpha_{0} \subset \Gamma_{y_{0},l_{k}}^{g_{m+1}}$\,
for \,$k >>0$.
We can apply the induction hypothesis to
\,$y_{0} \in \mathrm{Fix}( G_{(1)}, g_{0},\hdots, g_{m} ) - \mathrm{Fix}(g_{m+1})$.
Analogously we obtain a point
\,$y_{1} \in \mathrm{Fix}(A) - \mathrm{Fix}(f)$\,
that is $\omega$-recurrent for \,$f$\, and a simple closed curve
\,$\alpha_{1} \subset \Gamma_{\!\!y_{1}}^{f}$\, contained in
\,$\mathrm{Int}({\Delta}_{0})$.
Let \,$\Delta_{1}$\, be the compact disc bounded by \,$\alpha_{1}$.
By repeating this process we obtain a sequence
\,${\mathcal D}_{k} \supset \Delta_{0} \supset \Delta_{1} \supset \cdots \supset \Delta_{n} \supset
\cdots$\,
such that \,$\Delta_{j+1} \subset \mathrm{Int}({\Delta}_{j})$\,
for any \,$j \geq 0$\, where \,$\partial \Delta_{j} \subset \Gamma_{\!\!y_{j}}^{g_{m+1}}$\,
for some \,$y_{j} \in \mathrm{Fix}(B) - \mathrm{Fix}(g_{m+1})$\, if \,$j$\, is even
and \,$\partial \Delta_{j} \subset \Gamma_{\!\!y_{j}}^{f}$\,
for some \,$y_{j} \in \mathrm{Fix}(A) - \mathrm{Fix}(f)$\, if \,$j$\, is odd.

We have \,$\mathrm{Fix}(A\,,f) \cap \mathrm{Int}({\mathcal D}_{k}) = \emptyset$\,
and \,$\mathrm{Fix}(f) \cap \gamma_{k}=\emptyset$.
We obtain
\,$\mathrm{Fix} (A\,,f) \cap {\mathcal D}_{k} = \emptyset$\, and
\,$\mathrm{Fix} (B\,,g_{m+1}) \cap {\mathcal D}_{k} = \emptyset$.
Therefore there exists \,$r \in {\mathbb R}^{+}$\, such that
\,$\|f(y)-y\| \geq r$\, for any \,$y \in \mathrm{Fix}(A) \cap {\mathcal D}_{k}$\, and
\,$\|g_{m+1}(y)-y\| \geq r$\, for any \,$y \in \mathrm{Fix}(B) \cap {\mathcal D}_{k}$.
Notice that \,${\mathcal O}_{y_{j}}(g_{m+1}) \subset \mathrm{Fix}(B) \cap {\mathcal D}_{k}$\,
if \,$j$\, is even
and \,${\mathcal O}_{y_{j}}(f) \subset \mathrm{Fix}(A) \cap {\mathcal D}_{k}$\, if \,$j$\, is odd.
Hence the item $(ii)$ of Lemma \ref{k:gr:kapa} implies that
\,$d(\partial \Delta_{j}, \partial \Delta_{j+1}) \geq r$\, for any \,$j \geq 0$.
We deduce that there exists a ball of radius \,$r/3$\, contained in
\,$\mathrm{Int}(\Delta_{j}) - \Delta_{j+1}$\, for any \,$j \geq 0$.
Since the area of \,$\Delta_{0}$\, is finite we obtain a contradiction.
\end{proof}

Lemma \ref{lem:fdertog} is used to find global fixed points for normal subgroups
of \,$G$. It is not clear that the $G$-orbit of such points is bounded.
This issue motivates the definition of capital points
since a capital
point for an orbit \,${\mathcal O}_{p}(f)$\, has bounded $G$-orbit under suitable conditions.
Moreover, the sets of capital points are naturally invariant by
Lemma \ref{ind:pt:02}. The existence of capital points is
guaranteed by next two lemmas.

\begin{lem}
\label{lem:index}
Let \,$G \subset \mathrm{Homeo}({\mathbb R}^{2})$\, be a $\sigma$-step nilpotent group
finitely generated in \,${\mathcal V}_{\sigma}$.
Let \,$g_{1}, \ldots, g_{m}, f \in {\mathcal V}_{\sigma} \cap G$\, and
\,$p \in \mathrm{Fix} (G_{(1)}, g_{1},\hdots,g_{m}) - \mathrm{Fix}(f)$.
Suppose there exists an increasing sequence \,$(n_{k})_{k \geq 1}$\,
of positive integers such that \,$f^{n_{k}}(p) \to p$.
Then there exists
\[ q_{k} \in \mathrm{Fix} (G_{(1)}, g_{1},\hdots,g_{m},f)  \quad \text{such that} \quad
\mathrm{Ind}_{q_{k}}(\Gamma_{\!\!p,n_{k}}^{f}) \neq 0 \quad \text{for} \quad k>>0\,. \]
\end{lem}

\begin{proof}
The angle in between two oriented non-disjoint
segments in \,$\Gamma_{\!\!p}^{f}$\, is less than \,$\pi /3$\, by
Corollary \ref{co:dpi}.
The segment \,$[\,f^{n_{k}}(p)\,,f(p)\,]$\, is very close to
\,$[\,p\,,f(p)\,]$\, if \,$k >>0$. Thus
the angle in between two oriented non-disjoint
segments in \,$\Gamma_{\!\!p,n_{k}}^{f}$\, is less than \,$\pi /2$\, for \,$k >>0$.
Lemma 3.1 of \cite{fi05}
implies that there exists a simple closed curve \,$\gamma_{k} \subset \Gamma_{\!\!p,n_{k}}^{f}$\,
that is oriented as \,$\Gamma_{\!\!p,n_{k}}^{f}$\, and satisfies
\,$\mathrm{Ind}_{q}(\Gamma_{\!\!p,n_{k}}^{f}) \neq 0$\, for any \,$q \in \mathrm{Int}({\mathcal D}_{k})$\, where
\,${\mathcal D}_{k}$\, is the compact disc bounded by \,$\gamma_{k}$.
There exists a point
\,$q_{k} \in  \mathrm{Fix} (G_{(1)}, g_{1},\ldots,g_{m},f) \cap \mathrm{Int}({\mathcal D}_{k})$\,
by Lemma \ref{lem:fdertog} for \,$k >>0$.
Clearly we obtain \,$\mathrm{Ind}_{q_{k}}(\Gamma_{\!\!p,n_{k}}^{f}) \neq 0$.
\end{proof}

\begin{lem}
\label{lem:clcv}
Let \,$G \subset \mathrm{Homeo}({\mathbb R}^{2})$\, be a $\sigma$-step nilpotent group
finitely generated in \,${\mathcal V}_{\sigma}$.
Let \,$g_{1}, \ldots, g_{m}, f \in {\mathcal V}_{\sigma} \cap G$. Suppose that
\,$p \in \mathrm{Fix} (G_{(1)}, g_{1},\hdots,g_{m})$\, has bounded $f$-orbit.
Then there exists \,$q \in \mathrm{Fix} (G_{(1)}, g_{1},\ldots,g_{m},f)$\, satisfying
one of the following conditions$\,:$
\begin{itemize}
\item
$q \in \overline{{\mathcal O}_{p}(f)} \,;$
\item
$q$\, is a capital point for \,${\mathcal O}_{\tilde{p}}(f)$\, where
\,$\tilde{p} \in \overline{{\mathcal O}_{p}(f)}$\,.
\end{itemize}
\end{lem}

\begin{proof}
If \,$\overline{{\mathcal O}_{p}(f)} \cap \mathrm{Fix}(f) \neq \emptyset$\,
there is nothing to prove. Let us suppose
\,$\overline{{\mathcal O}_{p}(f)} \cap \mathrm{Fix}(f) = \emptyset$\, and
let \,$\tilde{p}\notin \text{Fix}(f)$\, be a $\omega$-recurrent point for \,$f$\, in \,$\overline{{\mathcal O}_{p}(f)}$.
Consider an increasing sequence \,$(n_{k})_{k \geq 1}$\,  of positive integers such that
\,$f^{n_{k}}(\tilde{p}) \to \tilde{p}$.
There exists \,$q_{k} \in \mathrm{Fix} (G_{(1)}, g_{1},\ldots,g_{m},f)$\,
such that \,$\mathrm{Ind}_{q_{k}}(\Gamma_{\!\!\tilde{p},n_{k}}^{f}) \neq 0$\,
for \,$k$\, big enough by Lemma \ref{lem:index}.
The condition \,$\mathrm{Ind}_{q_{k}}(\Gamma_{\!\!\tilde{p},n_{k}}^{f}) \neq 0$\,
implies that \,$q_{k}$\, belongs to \,$\mathrm{Conv}({\mathcal O}_{\tilde{p}}(f))$.
Since \,${\mathcal O}_{\tilde{p}}(f) \subset  \overline{{\mathcal O}_{p}(f)}$\,
the sequence \,$(q_{k})_{k \geq 1}$\, is bounded.
Up to consider a subsequence we can suppose \,$q_{k} \to q$.
Lemma \ref{lem:iso} provides a positive radius \,$\delta$\, such that
\,$B(q\,,\delta) \cap \Gamma_{\!\!\tilde{p},n_{k}}^{f} = \emptyset$\, for
\,$k >>0$.
Since \,$\mathrm{Ind}_{q}(\Gamma_{\!\!\tilde{p},n_{k}}^{f}) = \mathrm{Ind}_{q_{k}}(\Gamma_{\!\!\tilde{p},n_{k}}^{f})$\,
for \,$k >>0$\, then \,$q$\, is a capital point for
\,${\mathcal O}_{\tilde{p}}(f)$.
\end{proof}

\vskip10pt

Let \,$G$\, be a nilpotent subgroup of \,$\mathrm{Homeo}({\mathbb R}^{2})$\,
generated by $\epsilon$-Lipschitz homeomorphisms
with respect to the identity.
Our goal is proving that if the group has a bounded orbit \,${\mathcal O}$\, then
either there exists a global fixed point in the closure of \,${\mathcal O}$\, or
\,${\mathcal O}$\, encloses a global fixed point in the following sense\,: there exists a
global fixed point in the interior of  \,$\mathrm{Conv}({{\mathcal O}})$.
It is analogous in our setting to a result for abelian groups of $C^{1}$-diffeomorphisms
by Franks, Handel and Parwani \cite[Theorem 6.1]{fhp01}.

\vskip10pt

\paragraph{\bf Definitions.}
Let \,$G$\, be a subgroup of \,$\mathrm{Homeo}({\mathbb R}^{2})$\, and \,$p \in {\mathbb R}^{2}$\,
be a point whose $G$-orbit is bounded.
Given an orientation-preserving homeomorphism  \,$f \in G$\, we define\,:
$$\mathcal{A}_{G}(\,p\,,f\,) : =  \bigcup_{y \in \overline{{\mathcal  O}_{p}(G)}} \ [\,y\,,f(y)\,]
\quad \text{and} \quad  \epsilon_{G}(\,p\,,f\,)=
d \big(\mathcal{A}_{G}(p\,,f) \,, \mathrm{Fix}(f)\big) \,.$$
We also define \,$\mathcal{B}_{G}(\,p\,,f\,)$\, as the union of
\,$\mathcal{A}_{G}(\,p\,,f\,)$\, and the bounded connected components of
\,${\mathbb R}^{2} - \mathcal{A}_{G}(\,p\,,f\,)$.

\vskip10pt

Let us explain the idea behind the above definitions.
The segment \,$[\,f^{n_{k}}(z)\,,f(z)\,]$\, converges to the segment
\,$[\,z\,,f(z)\,]$\, when \,$f^{n_{k}}(z) \rightarrow z$\,.
As a consequence if
\,$f \in \mathcal{U}$\, and \,$\overline{\mathcal{O}_{p}(G)} \cap \mathrm{Fix}(f)=\emptyset$\,
then the capital points for \,$\mathcal{O}_{z}(f)$\, belong to the bounded connected
components of \,$\RR-\mathcal{A}_{G}(\,p\,,f)$\, when \,$z\in \overline{\mathcal{O}_{p}(G)}$\,.
Thus the set \,$\mathcal{B}_{G}(\,p\,,f\,)$\, is a natural place to localize
global fixed points.

\vskip10pt

\begin{remark}
\label{rem:nfp}
Consider the setting in the above definitions. The set \,$\mathcal{A}_{G}(\,p\,,f\,)$\, is compact. Moreover
if \,$f \in {\mathcal U}$\, and
\,$\overline{{\mathcal  O}_{p}(G)} \cap \mathrm{Fix}(f) = \emptyset$\,
then \,$\epsilon_{G}(\,p\,,f\,)$\, is greater than \,$0$\, by Corollary \ref{co:dpi}\,
and we
have that \,$B(z\,,\epsilon_{G}(\,p\,,f\,))$\, is contained in
\,$\RR-\mathcal{A}_{G}(\,p\,,f\,)$\, for all \,$z\in \mathrm{Fix}(f)$\,. In particular if
\,$z\in \mathrm{Fix}(f)$\, is not contained in the unbounded connected componente of
\,$\RR-\mathcal{A}_{G}(\,p\,,f\,)$\, then \,$z\in \mathcal{B}_{G}(\,p\,,f\,)$\, and consequently
\,$B(\,z\,,\epsilon_{G}(\,p\,,f\,))$\, is contained in
\,$\mathcal{B}_{G}(\,p\,,f\,)$\,.

\end{remark}

\vskip5pt

Given \,$A\subset \RR$\, and \,$r>0$\, we define
\,$B(A\,,r)=\cup_{z\in A}B(z\,,r)$\,.

\begin{pro}
\label{pro:im}
Let \,$G \subset \mathrm{Homeo}({\mathbb R}^{2})$\, be a finitely generated
$\sigma$-step nilpotent subgroup generated
in \,${\mathcal V}_{\sigma}$.
Consider a normal subgroup \,$H$\, of \,$G$\, generated in \,${\mathcal V}_{\sigma}$\,   and
\,$f \in G \cap {\mathcal V}_{\sigma}$\, such that the class of \,$f$\, in \,$G/H$\, belongs to
\,$Z(G/H)$\,.
Let \,$p \in \mathrm{Fix}(H)$\, with bounded $G$-orbit.
Then there exists  \,$q\in \mathrm{Fix}(H,f)$\,
such that either
\[ q \in \overline{{\mathcal O}_{p}(G)} \ \ \ \text{or} \ \ \
\epsilon_{G}(\,p\,,f\,)>0 \ \ \ \text{and} \ \ \
B\big(\mathrm{Conv}(\overline{{\mathcal O}_{q}(G)})\,,\epsilon_{G}(\,p\,,f\,)\big) \subset
\mathrm{Conv}(\overline{{\mathcal O}_{p}(G)})\,. \]
We always have
\, ${\mathcal O}_{q}(G) \subset \mathrm{Conv}(\overline{{\mathcal O}_{p}(G)})$\,.
\end{pro}

\begin{proof}
We denote \,$J=\langle H,f \rangle$.
If \,$\overline{{\mathcal  O}_{p}(G)} \cap \mathrm{Fix}(f) \neq \emptyset$\, we
choose \,$q \in \overline{{\mathcal  O}_{p}(G)} \cap \mathrm{Fix}(f)$.
It is clear that \,$q\in\mathrm{Fix}(H,f)$\, and \,$q \in \overline{{\mathcal  O}_{p}(G)}$\, by construction.

Let us suppose that \,$\overline{{\mathcal  O}_{p}(G)} \cap \mathrm{Fix}(f) = \emptyset$.
We deduce
\,$\mathcal{A}_{G}(\,p\,,f\,) \cap \mathrm{Fix}(f) = \emptyset$\, and \,$\epsilon_{G}(\,p\,,f\,)>0$\,
by Remark \ref{rem:nfp}.

The group \,$G$\, is a finitely generated nilpotent group and hence polycyclic
(cf. \cite{Karga}[Theorem 17.2.2]).
Therefore any subgroup of \,$G$\,  is finitely generated
(cf. \cite{Karga}[Theorem 19.2.3]) and we obtain
\,$H = \langle J_{(1)}, g_{1},\hdots,g_{m} \rangle$\,
for some \,$g_{1},\ldots,g_{m} \in {\mathcal V}_{\sigma}$.

We remind the reader that \,$\overline{{\mathcal  O}_{p}(G)} \cap \mathrm{Fix}(J) = \emptyset$\,.
Now, applying Lemma \ref{lem:clcv} to \,$J\hskip1pt,g_{1},\ldots,g_{m},f$\, and \,$p\in \mathrm{Fix}(J_{(1)},g_{1},\ldots,g_{m})=\mathrm{Fix}(H)$\, we conclude there exists \,$q\in \mathrm{Fix}(J)$\, that is a capital point associated to
\,$\mathcal{O}_{\tilde{p}}(f)$\, for some \,$\tilde p \in \overline{{\mathcal O}_{p}(f)}$\,.

Let \,$\{h_{1},\hdots,h_{n}\} \subset {\mathcal V}_{\sigma}$\, be a set such that
\,$\langle H,f,h_{1},\hdots,h_{n} \rangle=G$.
Since
\[ \overline{{\mathcal O}_{q}(h_{1},\hdots,h_{n})} \cap
\overline{{\mathcal O}_{\tilde{p}}(h_{1},\hdots,h_{n},f)} \subset
\mathrm{Fix}(J) \cap \overline{{\mathcal  O}_{p}(G)} = \emptyset \]
we can apply Lemma \ref{fechopontocap}.
Let \,$q' \in \overline{{\mathcal O}_{q}(h_{1},\hdots,h_{n})}=\overline{{\mathcal  O}_{q}(G)}$.
The point \,$q'$\, is a capital point associated to \,${\mathcal O}_{z}(f)$\, for some
\,$z \in {\mathcal O}_{\tilde{p}}(h_{1},\hdots,h_{n})$.
Therefore \,$q'$\, belongs to \,$\mathcal{B}_{G}(\,z\,,f\,)$\, and then to
\,$\mathcal{B}_{G}(\,p\,,f\,)$\,
since \,$z \in \overline{{\mathcal  O}_{p}(G)}$.
From Remark \ref{rem:nfp} we deduce that
\,$B(q',\epsilon_{G}(\,p\,,f\,)) \subset \mathcal{B}_{G}(\,p\,,f\,)$. We obtain
\[ B\big(\overline{{\mathcal  O}_{q}(G)}\,, \epsilon_{G}(\,p\,,f\,)\big) \subset \mathcal{B}_{G}(\,p\,,f\,)
\subset \mathrm{Conv}(\overline{{\mathcal  O}_{p}(G)}) \]
as we wanted to prove.
\end{proof}

Finally we can complete the proof of the Main Theorem.

\begin{defn*}
\label{def:central}
Let \,$G$\, be a  group and \,$f_{0}=Id$\,.
We say that the sequence \,$f_{1},\ldots,f_{n}$\,  is a \,{\it central series}\, for
\,$G$\,
if\,:
\begin{itemize}
\item
$G=\langle f_{1},\hdots,f_{n} \rangle$\,;

\item
$\langle f_{0},\hdots,f_{j} \rangle$\, is a normal subgroup of
\,$G$\, for any \,$0 \leq j \leq n$\,
and the class of \,$f_{j+1}$\, belongs to the center of \,$G/\langle f_{0},\hdots,f_{j} \rangle$\,
for any \,$0 \leq j < n$.

\end{itemize}

Let \,$\mathcal{S}$\, be a set of generators for \,$G$.
We say that a central series \,$f_{1}, \ldots, f_{n}$\, is   associated to
\,${\mathcal S}$\, if \,$\{f_{1},\hdots,f_{n}\} \subset \cup_{j=0}^{\infty} \, {\mathcal S}_{(j)}$.
\end{defn*}

\begin{remark}
\label{rem:central}
Let \,$G$\, be a $\sigma$-step nilpotent subgroup of \,$\mathrm{Homeo}({\mathbb R}^{2})$\,
generated by a finite subset \,${\mathcal S}$.
The sequence \,${\mathcal S}_{(\sigma-1)}, \ldots, {\mathcal S}_{(0)}$\,
(we can choose any order in each \,${\mathcal S}_{(j)}$)
is a central series
associated to \,${\mathcal S}$\, by Lemma \ref{lem:censer}.
\end{remark}

\begin{theo}
\label{teo:mai}
Let \,$G \subset \mathrm{Homeo}({\mathbb R}^{2})$\, be a $\sigma$-step nilpotent subgroup
generated in \,${\mathcal V}_{\sigma+1}$.  Let \,$p \in {\mathbb R}^{2}$\, with  bounded $G$-orbit.
Then there exists \,$q \in \mathrm{Fix}(G)$\,
such that either \,$q \in \overline{{\mathcal O}_{p}(G)}$\,
or \,$q \in \mathrm{Int}(\mathrm{Conv}({\mathcal O}_{p}(G)))$.
\end{theo}

\begin{proof}
Suppose that \,$G$\, is finitely generated.
Let \,${\mathcal S}$\, be a finite generator set of \,$G$\, contained in \,${\mathcal V}_{\sigma+1}$.
There exists
a central series \,$\{g_{1},\ldots,g_{n}\}$\, associated to \,${\mathcal S}$\,
and contained in \,${\mathcal V}_{\sigma}$\, by Remarks \ref{rem:loss} and
\ref{rem:central}. From Proposition \ref{pro:im}
there exists
a sequence
$$p_{0}=p \ , \ p_{i} \in \mathrm{Fix}(g_{1},\hdots,g_{i}) \quad \text{for} \quad
1 \leq i \leq n $$
such that
either
$$p_{i+1} \in \overline{{\mathcal O}_{p_{i}}(G)} \quad
\text{or} \quad \overline{{\mathcal O}_{p_{i+1}}(G)} \subset
\mathrm{Int}(\mathrm{Conv}(\overline{{\mathcal O}_{p_{i}}(G)})) \quad \text{for any} \quad
0\leq i < n \,.$$
Now we define \,$q:=p_{n}$\, and we obtain after an easy calculation that either
\[ q \in \overline{{\mathcal O}_{p}(G)} \quad  \mathrm{or} \quad
q \in
\mathrm{Int}(\mathrm{Conv}(\overline{{\mathcal O}_{p}(G)}))=\mathrm{Int}(\mathrm{Conv}({\mathcal O}_{p}(G))) . \]

Let us consider the general case. For this let \,${\mathcal S} \subset {\mathcal V}_{\sigma+1}$\,
be an (infinite) generator set of \,$G$. We denote
\[ j = \min \big\{ l \in {\mathbb N} \cup \{0\} \ ; \  \overline{{\mathcal O}_{p}(G)}
\cap \mathrm{Fix}(G_{(l)}) \neq \emptyset \big\}. \]
If \,$j=0$\, we are done. Let us suppose \,$j\geq 1$\, and recall that
\,$G_{(l)} = \langle {\mathcal S}_{(l)}, \hdots, {\mathcal S}_{(\sigma -1)} \rangle$\, for
\,$l \geq 0$\,
by Lemma \ref{lem:censer}.

Given \,$f \in {\mathcal S}_{(j-1)}$\, we define the open set
\,$\mathfrak{U}_{f}=\{ y \in {\mathbb R}^{2} \ ; \, f(y) \neq y \}$. We obtain
\[ \overline{{\mathcal O}_{p}(G)} \cap \mathrm{Fix}(G_{(j)}) \subset
\bigcup_{f \in {\mathcal S}_{(j-1)}} \hskip-7pt \mathfrak{U}_{f} \]
since \,$G_{(j-1)}=\langle {\mathcal S}_{(j-1)}, G_{(j)} \rangle$\, and
\,$\overline{{\mathcal O}_{p}(G)}
\cap \mathrm{Fix}(G_{(j-1)}) = \emptyset $\,.

By compactness of  \,$\overline{{\mathcal O}_{p}(G)} \cap \mathrm{Fix}(G_{(j)})$\,
there exists a subset \,$\{ f_{1}, \ldots, f_{a} \}$\,
of \,${\mathcal S}_{(j-1)}$\, such that
\,$\overline{{\mathcal O}_{p}(G)} \cap \mathrm{Fix}(G_{(j)}) \subset \mathfrak{U}_{f_{1}} \cup \ldots \cup \mathfrak{U}_{f_{a}}$.
If \,$\{ f_{1}, \ldots, f_{a} \}$\, is minimal among the sets sharing the previous property
we obtain
\begin{align}\label{rel:imp:fin}
 \overline{{\mathcal O}_{p}(G)} \cap \mathrm{Fix}(G_{(j)},f_{0}, \ldots,f_{a-1}) \neq \emptyset
\ \ \ \mathrm{and} \ \ \
\overline{{\mathcal O}_{p}(G)} \cap \mathrm{Fix}(G_{(j)},f_{1},\ldots,f_{a}) = \emptyset
\end{align}
where \,$f_{0}=Id$\,.
We recall that \,$\langle  G_{(j)} \,, \mathcal{T} \, \rangle$\, is a normal subgroup of \,$G$\, whenever \,$\mathcal{T}$\, is a subset of \,$G_{(j-1)}$\,.

Let \,$L$\, be a finitely generated subgroup of \,$G$.
There exists a finite subset \,$\widehat{\mathcal S}$\, of \,${\mathcal S}$\,
such that \,$L \subset \langle \widehat{\mathcal S} \,\rangle$\, and
\,$f_{1},\ldots,f_{a} \in \widehat{\mathcal S}_{(j-1)}$. We denote \,$M=\langle \widehat{\mathcal S} \, \rangle$.
There exists a central series
\[ g_{1},\ldots,g_{n}, f_{1},\ldots,f_{m} \in {\mathcal V}_{\sigma} \]
associated to \,$\widehat{\mathcal S}$\,
with \,$m \geq a$\, and
\,$\{g_{1},\hdots,g_{n}\} \subset G_{(j)}$\, by Remark \ref{rem:central}.
We define
\[  H=\langle g_{1},\hdots,g_{n}, f_{0},f_{1},\hdots,f_{a-1} \rangle  \]
It is a normal subgroup of \,$M$\, contained in
\,$G_{(j-1)}$.

Now let us fix \,$\hat p_{0} \in \overline{{\mathcal O}_{p}(G)}
\cap \mathrm{Fix}(G_{(j)},f_{0},f_{1},\hdots,f_{a-1})$.
We have that \,$\mathcal{O}_{\hat p_{0}}(M)$\, is bounded and
\,$\hat p_{0} \in \text{Fix}(H)$\,.
It follows from   \eqref{rel:imp:fin} that
\,$\overline{{\mathcal O}_{\hat p_{0}}(G)} \cap \text{Fix}(f_{a})=\emptyset$\, since
$$  \overline{{\mathcal O}_{\hat p_{0}}(M)}
\subset \overline{{\mathcal O}_{\hat p_{0}}(G)}
 \subset  \overline{{\mathcal O}_{p}(G)}
\cap \mathrm{Fix}(G_{(j)},f_{0},f_{1},\hdots,f_{a-1}) \,.$$
Applying Proposition \ref{pro:im} with \,$f=f_{a}$\, and \,$\hat p_{0} \in \text{Fix}(H)$\, we conclude that
$$0 < \epsilon_{G}(\hat p_{0}\,,f_{a})\leq \epsilon_{M}(\hat p_{0}\,,f_{a})$$
and there exists a point  \,$\hat p_{a} \in  \mathrm{Fix}(H,f_{a})$\, such that
\begin{align}\label{fin:bol:co}
  B\big(\mathrm{Conv}(\overline{{\mathcal O}_{\hat p_{a}}(M)}),\epsilon_{G}(\hat p_{0}\,,f_{a})\big)  & \subset
B\big(\mathrm{Conv}(\overline{{\mathcal O}_{\hat p_{a}}(M)}),\epsilon_{M}(\hat p_{0},f_{a})\big) \notag \\
& \subset
 \mathrm{Conv}\big(\overline{{\mathcal O}_{\hat p_{0}}(M)}\big) \subset
\mathrm{Conv}\big(\overline{{\mathcal O}_{p}(G)}\big).
\end{align}
By successive applications of Proposition \ref{pro:im} we obtain
$$\hat p_{a+1} \in \mathrm{Fix}(H,f_{a}\,,f_{a+1})\,,\ldots,\,\hat p_{m} \in \mathrm{Fix}(H,f_{a}\,,\hdots,f_{m})$$
such that
\[ \mathrm{Conv} (\overline{{\mathcal O}_{\hat p_{m}}(M)}) \subset \cdots \subset
\mathrm{Conv} (\overline{{\mathcal O}_{\hat p_{a}}(M)})\,.  \]

\noindent
The point \,$\hat p_{m}\in \text{Fix}(L)$\, belongs to
\,$\mathrm{Fix}(M) \cap \mathrm{Conv} (\overline{{\mathcal O}_{\hat p_{a}}(M)})$\,.
Let us define \,$\epsilon := \epsilon_{G}(\,\hat p_{0}\,,f_{a})$.
Then it follows from \eqref{fin:bol:co} that
\,$B(\hat p_{m},\epsilon) \subset \mathrm{Conv}(\overline{{\mathcal O}_{p}(G)})
=\overline{\mathrm{Conv}({\mathcal O}_{p}(G)})$.
Moreover we obtain \,$B(\hat p_{m},\epsilon) \subset \mathrm{Conv}({\mathcal O}_{p}(G))$\,
since the interiors of \,$\mathrm{Conv}(\overline{{\mathcal O}_{p}(G)})$\,
and \,$\mathrm{Conv}({\mathcal O}_{p}(G))$\, coincide.

Now let us consider the compact set
\[ \mathfrak{F} = \big\{ y \in \overline{\mathrm{Conv}({\mathcal O}_{p}(G)}) \ \,; \
d\big(\,y\,,\partial \,\mathrm{Conv}({\mathcal O}_{p}(G))\,\big) \geq \epsilon  \big\} . \]
By construction every finitely generated subgroup \,$L$\, of \,$G$\, has a global fixed point
in \,$\mathfrak{F}$. Hence \,$\mathrm{Fix}(G) \cap \mathfrak{F}$\, is non-empty by the
finite intersection property.
Consequently the ball \,$B(q\,,\epsilon)$\, is contained in
\,$\mathrm{Conv}({\mathcal O}_{p}(G))$\,
and then
\,$q \in \mathrm{Int}(\mathrm{Conv}({\mathcal O}_{p}(G)))$\,
for any \,$q \in \mathrm{Fix}(G) \cap \mathfrak{F}$.
\end{proof}

\vskip10pt

\begin{remark}
A goal of this paper is showing existence and localization of global
fixed points of nilpotent groups within an elementary framework.
In this spirit we try to prevent technical difficulties that could
make the paper more difficult to read.
For instance in Theorem \ref{teorema} the localization result suggests
that it suffices to require the
Lipschitz condition in a big enough neighborhood of the bounded orbit.
This is indeed the case but we prefer to avoid the extra notations and
estimates whereas we keep the main ideas.

Another example is the definition of the constants \,$\epsilon_{\sigma}$\,
and the sets \,${\mathcal V}_{\sigma}$\, of homeomorphisms.
It is possible to show that we can choose the sequence
\,$(\epsilon_{\sigma})_{\sigma \geq 0}$\, such that it converges
to $0$ geometrically by taking profit of the properties of central series.
The proof involves a tighter control of the inductive process.
Again we prefer a neat simple approach.
\end{remark}


\vglue30pt
\section{Nilpotent actions and Cartwright-Littlewood Theorem}
\vskip10pt


The next result immediately implies Theorem \ref{CLnilp} if the nilpotent group
is finitely generated.

\vskip10pt
\begin{theo}\label{CLnilpfg}
Let \,$G\subset \mathrm{Homeo}(\RR)$\, be a $\sigma$-step nilpotent group
finitely generated in \,$\mathcal{V}_{\sigma+1}$.
If \,$\mathcal{C}$\, is a  \,$G$-invariant full continuum  then
there is a global fixed point of  \,$G$\, in \,$\mathcal{C}$.
\end{theo}


\begin{proof}
Let \,$\{g_{1},\ldots,g_{n}\}$\, be a central sequence associated to \,${\mathcal S}$,
where \,${\mathcal S} \subset {\mathcal V}_{\sigma +1}$\, is a finite generator set of \,$G$.
It is contained in \,${\mathcal V}_{\sigma}$\, by Remark \ref{rem:loss}.
We denote \,$G_{l}=\langle g_{1},\hdots,g_{l} \rangle$\, and \,$K_{l}=\mathrm{Fix}(G_{l}) \cap {\mathcal C}$.
Let us show that \,$K_{j} \neq \emptyset$\,
for any \,$1 \leq j \leq n$\, by induction on \,$j$.
The statement holds true for \,$j=1$\, by Cartwright-Littlewood theorem.

We will show that \,$K_{j} \neq \emptyset$\, and \,$K_{j+1} = \emptyset$\,
are incompatible.
Let \,$\mathcal{W}$\, be the connected component of
\,$\RR-\mathrm{Fix}(G_{j+1})$\, containing \,$\mathcal{C}$.
Since \,$\mathcal{C}$\, is \,$G_{j+1}$-invariant we deduce that
so is \,${\mathcal W}$.
Moreover, the set \,$K_{j}$\, is \,$g_{j+1}$-invariant by Remark \ref{inv:fix}, thus there exists a
$\omega$-recurrent point \,$p$\, for \,$g_{j+1}$\, in \,$K_{j}$.
Let \,$(n_{k})_{k \geq 1}$\,  be an increasing sequence of positive integers such that
\,$g_{j+1}^{n_{k}}(p) \to p$.
There exists \,$q_k \in \mathrm{Fix}((G_{j+1})_{(1)},g_1,\ldots,g_{j+1})$\, such that
\,$\mathrm{Ind}_{q_k}(\Gamma^{g_{j+1}}_{\!\!p,n_k}) \neq 0$\, for \,$k>>0$\,
by Lemma \ref{lem:index}.
In particular the curves \,$\Gamma^{g_{j+1}}_{\!\!p,n_k}$\, are not null-homotopic
in \,${\mathcal W}$\, for \,$k>>0$.

Let us show that there is a lift of \,$\Gamma^{g_{j+1}}_{\!\!p,n_k}$\, to the universal
covering \,$\pi : {\mathbb R}^{2} \to {\mathcal W}$\, of \,${\mathcal W}$\, that is
a closed curve for \,$k >>0$, contradicting that \,$\Gamma^{g_{j+1}}_{\!\!p,n_k}$\, is
non-null-homotopic.
Let \,$D$\, be a closed topological disc containing \,${\mathcal C}$\, and contained
in \,${\mathcal W}$. Consider a lift \,$\widetilde{D}$\, of \,$D$\, in \,${\mathbb R}^{2}$.
The set \,$\widetilde{\mathcal C}:= \widetilde{D} \cap \pi^{-1} ({\mathcal C})$\, is a lift
of \,${\mathcal C}$. Given \,$q \in {\mathcal C}$\, we denote by \,$\tilde{q}$\, the
unique point in \,$\widetilde{\mathcal C} \cap \pi^{-1}(q)$.

Consider the isotopy \,$H_{t}(x):=(1-t) x + t \, g_{j+1}(x)$\,
given by Corollary \ref{isotopia}
where \,$x \in {\mathbb R}^{2}$\, and \,$t \in [0,1]$.
Since it is an isotopy relative to \,$\mathrm{Fix}(g_{j+1})$\,
it can be restricted to an isotopy in \,${\mathcal W}$.
We consider the lift \,$\widetilde{H}_{t}: [\,0\,,\,1] \times {\mathbb R}^{2} \to {\mathbb R}^{2}$\,
such that $\widetilde{H}_{0} \equiv Id$.
Given any point \,$y \in {\mathbb R}^{2}$\, the path  \,$\widetilde{H}_{t}(y)$\, where
\,$t \in [\,0\,,1\,]$\, is a lift of the segment \,$[\,\pi(y)\,, g_{j+1}(\pi(y))\,]$\, contained in
\,${\mathcal W}$.
We denote \,$\tilde{g}_{j+1} \equiv \widetilde{H}_{1}$.
Since there exists \,$p_{0} \in \mathrm{Fix}(g_{j+1}) \cap {\mathcal C}$\,
by Cartwright-Littlewood theorem
we obtain \,$\tilde{g}_{j+1}(\tilde{p}_{0})=\tilde{p}_{0}$\, and then
\,$\tilde{g}_{j+1}(\widetilde{\mathcal C}\,)=\widetilde{\mathcal C}$.
We deduce that there exists a lift of the curve \,$\Gamma_{\!\!p}^{g_{j+1}}$\,
whose vertices belong to \,$\widetilde{\mathcal C}$.
The property \,$g_{j+1}^{n_{k}}(p) \to p$\, implies that
any lift of the curve \,$\Gamma_{\!\!p,n_{k}}^{g_{j+1}}$\,
is a closed curve for \,$k >>0$.
\end{proof}

\vskip10pt

\vskip10pt
\subsection{Proof of Theorem \ref{CLnilp}}

The family composed by the sets of fixed points in \,${\mathcal C}$\, of finitely
generated subgroups of \,$G$\, is a family of compact sets that has the finite intersection property
by Theorem \ref{CLnilpfg}. Therefore the intersection of all sets in the family,
i.e. \,$\mathrm{Fix}(G) \cap {\mathcal C}$, is a non-empty set.


\vglue30pt
\appendix

\section{Revisiting Bonatti's ideas}\label{ap:b}
\vglue10pt


In order to show the existence of common fixed points for \,$C^1$-diffeomorphisms of the $2$-sphere, that are pairwise commuting and
$C^{1}$-close to the identity, Bonatti studies their local properties \cite{bo01}.
In this section we adapt these results for homeomorphisms of the plane that
are \,$\epsilon$-Lipschitz with respect to the identity.
The proofs are essentially the same as in  \cite{bo01} and
they are included in the paper for the sake of clarity.


\vskip10pt
\begin{lem}\label{lema2.1}
Let  \,$\epsilon>0$\,, \,$n\in\Z^{+}$\, and \,$f:\RR\rightarrow\RR$\, with
\,$\mathrm{Lip}(f-Id)\leq \epsilon/n$\,. Consider
\,$p\,,q\in\RR$\, such that \,$\|\,q-p\,\|\leq n \, \|\,f(p)-p\,\|$. Then we have
\begin{align}\label{lema:2.1}
 \|(f-Id)(q)-(f-Id)(p)\| \leq \epsilon \, \|f(p)-p\| \,.
\end{align}
In particular if \,$0<\epsilon < 1$\,  and  \,$f(p)\neq p$\, then
\begin{itemize}
\item[$(i)$]
\,$f$\, has no fixed points in the closed ball
\,$B\big[\,p\,, n\,\|f(p)-p\|\,\big]\,;$

\item[$(ii)$]
The angle \,$\mathrm{Ang}(v_{1},v_{2})$\, enclosed by the vectors
\begin{align}\label{va:pe}
 v_{1}=\frac{f(z_{1})-z_{1}}{\|f(z_{1})-z_{1}\|} \quad  \text{and} \quad
v_{2}=\frac{f(z_{2})-z_{2}}{\|f(z_{2})-z_{2}\|}     \end{align}
is well-defined if \,$z_{1}\,,z_{2}\in B\big[\,p\,, n\,\|f(p)-p\|\,\big]$\, and
\,$0\leq \mathrm{Ang}(v_{1},v_{2})\leq 2 \arcsin(\epsilon)$.
\end{itemize}

\end{lem}
\vskip5pt


\begin{proof}
Let  \,$p\,,q\in\RR$\, with  \,$\|q-p\|\leq n \, \|f(p)-p\|$.
The Lipschitz property implies
\[ \|\,(f-Id)(q)-(f-Id)(p)\,\| \leq  \frac{\epsilon}{n} \, \|\,q-p\,\|
\leq \frac{\epsilon}{n}\, n \, \|f(p)-p\|= \epsilon \, \|f(p)-p\| \]
proving the inequality \eqref{lema:2.1}.

Suppose \,$0<\epsilon < 1$\, and \,$f(p)\neq p$\,.
If \,$f(q)=q$\, for some \,$q\in B\big[\,p\,, n\,\|f(p)-p\|\,\big]$\, then
$$\|f(p)-p\|=\|\,(f-Id)(q)-(f-Id)(p)\,\|\leq \epsilon \,\|f(p)-p\|$$
contradicting the condition \,$0<\epsilon < 1$. This completes the proof of
item $(i)$.

Let us show item $(ii)$. The angle enclosed by the vectors \,$v_{1}$\, and \,$v_{2}$\,
is well-defined by item $(i)$. We have
$$\mathrm{Ang}(v_{1}\,,v_{2})\leq \mathrm{Ang}\Big(v_{1}\,,\frac{f(p)-p}{\|f(p)-p\|}\Big)+
\mathrm{Ang}\Big(v_{2}\,,\frac{f(p)-p}{\|f(p)-p\|}\Big).$$

\noindent
Moreover equation \eqref{lema:2.1} implies
\begin{align}\label{}
0\leq \sin\bigg\{\mathrm{Ang}\Big(v_{i}\,,\frac{f(p)-p}{\|f(p)-p\|}\Big)\bigg\}\leq
\frac{\|f(z_{i})-z_{i}-\big(f(p)-p\big)\|}{\|f(p)-p\|} \leq\epsilon \,.    \notag
\end{align}
Thus we obtain \,$0\leq \mathrm{Ang}(v_{1},v_{2})\leq 2 \arcsin(\epsilon)$, completing the proof of
item $(ii)$.
\end{proof}
\vskip5pt


\vskip10pt
\begin{cor}\label{co:dpi}
Let \,$f:\RR\rightarrow\RR$\, with \,$\mathrm{Lip}(f-Id)\leq \frac{1}{8}$\, and
\,$p\in\RR-\mathrm{Fix}(f)$\,. Then
\begin{itemize}
\item[$(i)$]
$f$\, is a homeomorphism$\,;$

\item[$(ii)$]
$ \|(f-Id)(q)-(f-Id)(p)\| \leq \frac{1}{2} \, \|f(p)-p\| $\,
for all
\,$p\,,q \in B[\,p\,,4\|f(p)-p\|\,]\,;$

\item[$(iii)$]
$f$\, has no fixed points in the closed ball \,$B[\,p\,,4\|f(p)-p\|\,]\,;$

\item[$(iv)$]
$0\leq \mathrm{Ang}(v_{1},v_{2})\leq\pi/3$\, where \,$v_{1}\,,v_{2}$\, are defined in \eqref{va:pe}
and \,$z_{1},z_{2}$\, belong to the closed ball \,$B[\,p\,,4\|f(p)-p\|\, ]$.
\end{itemize}
\end{cor}

\vskip5pt


\begin{proof}
The proof of items $(ii)$, $(iii)$ e $(iv)$ is obtained by considering
\,$\epsilon=1/2$\, and \,$n=4$\, in Lemma \ref{lema2.1}.
Item $(i)$ is an immediate consequence of Lemma \ref{elondifeo}.
\end{proof}

\vskip5pt

The next lemma is used in the proof of Lemma \ref{lem:fdertog}.


\vskip10pt
\begin{lem}\label{le:cs}
Let \,$f\in \mathrm{Homeo}(\RR)$\, with \,$\mathrm{Lip}(f-Id)\leq \frac{1}{8}$\,.
Suppose that \,$p\in\RR-\mathrm{Fix}(f)$\, is an \,$\omega$-recurrent point for \,$f$.
Then \,$\Gamma^{f}_{\!\!p}$\, is not a simple curve.
\end{lem}


\begin{proof}
Suppose that \,$\Gamma^{f}_{\!\!p}$\, is simple.
Let \,$\sigma$\, be a line segment, intersecting
\,$[\,p\,,f(p)\,]$\, perpendicularly in their common midpoint.
We also suppose that the length of \,$\sigma$\, is less or equal than \,$\|f(p)-p\|$.
Since \,$p$\, is $\omega$-recurrent the curve \,$\Gamma^{f}_{\!\!p}$\,
intersects \,$\sigma$\, infinitely many times.

Let \,$j \in\Z^{+}$\, such that
\,$\big[\,f^{j}(p)\,,f^{j+1}(p)\big] \cap \sigma \neq \emptyset$. We have
\[ \|f^{j}(p)-p\|  \leq \| f^{j+1}(p)-f^{j}(p)\| + 2 \|f(p)-p\|
\leq 3 \max\big\{ \|f(p)-p\|\,, \| f^{j+1}(p)-f^{j}(p)\| \big\}. \]
Hence either \,$f^{j}(p)\in B[\,p\,,4\,\|f(p)-p\| \,]$\, or
\,$p\in B[\,f^{j}(p)\,, 4\,\|f^{j+1}(p)-f^{j}(p)\|\,]$. Anyway, the angle
defined by the segments \,$[\,p\,,f(p)\,]$\, and \,$[\,f^{j}(p)\,,f^{j+1}(p)\,]$\, is less
or equal than \,$\pi / 3$\, by item $(iv)$ in Corollary \ref{co:dpi}.
Therefore these segments intersect \,$\sigma$\, with the same orientation.

Let \,$i\in\Z^{+}$\, be the first positive integer such that
\,$\big[\,f^{i}(p)\,,f^{i+1}(p)\big) \cap \sigma \neq \emptyset$.
Consider the simple closed curve \,$\beta$\, obtained by juxtaposing the segments
$$[\,a\,,f(p)\,]\ , \ [\,f(p)\,,f^{2}(p)\,] \ , \ \ldots \ , [\,f^{i-1}(p)\,,f^{i}(p)\,] \ , \
[\,f^{i}(p)\,,b\,] \ , \ [\,b\,,a\,]$$
where \,$\{a\} = [\,p\,,f(p)\,] \cap \sigma$\, and \,$\{b\} = \big[\,f^{i}(p)\,,f^{i+1}(p)\,\big) \cap \sigma$.
Since the segments \,$[\,p\,,f(p)\,]$\, and \,$\big[\,f^{i}(p)\,,f^{i+1}(p)\,\big)$\, intersect
\,$\sigma$\, with the same orientation then \,$\beta$\, separates the points \,$p$\, and
\,$f^{i+1}(p)$.

Let \,$\D$\, be the closure of the connected component of \,$\RR - \beta$\,
containing
\,$f^{i+1}(p)$.
Since \,$p$\, and \,$f^{i+1}(p)$\, belong to the $\omega$-limit of \,${\mathcal O}_{p}(f)$\,
there exists \,$q =f^{j}(p)$\, for some \,$j > i+1$\, such that \,$q \in \mathrm{Int}(\D)$\,
and \,$f(q)\notin \mathrm{Int}(\D)$\,.
The intersection of \,$[\,q\,,f(q)\,]$\, with \,$\beta$\, is contained in the segment \,$[\,a\,,b\,]$.
Since \,$[\,p\,,f(p)\,]$\, and \,$[\,q\,,f(q)\,]$\, intersect \,$\sigma$\, with the same orientation we
deduce \,$q$\, does not belong to \,$\mathrm{Int}(\D)$, obtaining a contradiction.
\end{proof}

\vglue5pt

Let \,$p\in\RR-\text{Fix}(f)$\, be an \,$\omega$-recurrent point for \,$f$.
There exists a simple closed curve \,$\gamma\subset \Gamma^{f}_{\!\!p}$\,
by Lemma \ref{le:cs}.
The vertices of \,$\gamma$\, are intersections of line segments of the form
\,$[\,f^{n}(p)\,,f^{n+1}(p)\,]$\, and \,$[\,f^{m}(p)\,,f^{m+1}(p)\,]$\,.
The angle described by two such intersecting segments is less or equal
than \,$\pi/3$\, by Corollary \ref{co:dpi}.
Let \,$\D$\, be the disc bounded by \,$\gamma$. Consider the vector field
defined by \,$X(x)=f(x)-x$\, for \,$x\in\RR$.
Since the singularities of \,$X$\, are the fixed points of \,$f$\, the vector field \,$X$\,
has no singular points in \,$\gamma$\, by Corollary \ref{co:dpi}.
Moreover, the angle described by \,$X(x)$\, and any segment of \,$\gamma$\, through \,$x$\,
is less or equal than \,$\pi/3$.
Therefore the index of the singularities of \,$X$\, in \,$\mathcal{D}$\, is equal to \,$1$.
We obtain the following lemma\,:


\vskip10pt
\begin{lem}\label{lem:ind:1}
Let \,$f \in \mathrm{Homeo}(\RR)$\, with \,$\mathrm{Lip}(f-Id)\leq \frac{1}{8}$\,.
Consider a $\omega$-recurrent point \,$p\in\RR-\mathrm{Fix}(f)$\, for \,$f$.
Then there exists a simple closed curve \,$\gamma\subset \Gamma^{f}_{\!\!p}$,
contained in \,$\RR - \mathrm{Fix}(f)$\,, such that the compact set of
fixed points of \,$f$\, in the interior of the disc bounded by \,$\gamma$\, has index \,$1$\, for
\,$f$.
\end{lem}


Analogously the existence of an increasing sequence of positive integers \,$(n_{k})_{k\geq1}$\,
such that \,$f^{n_{k}}(p) \rightarrow p$\, for some \,$p\in\RR-\mathrm{Fix}(f)$\,
implies that a simple closed curve  \,$\gamma_{k}\subset \Gamma^{f}_{p,n_{k}}$\,
behaves similarly as the curve \,$\gamma$\, in Lemma \ref{lem:ind:1} if \,$k >>0$.
The reason is that the segment \,$[\,f^{n_{k}}(p)\,,f(p)\,]$\, that is contained in
\,$\Gamma^{f}_{\!\!p,n_{k}}$\, but it is not necessarily contained in \,$\Gamma^{f}_{\!\!p}$\, has analogous
properties as the segment \,$[\,p\,,f(p)\,]$\, when \,$k >>0$.
Then we have the following lemma.


\vskip10pt
\begin{lem}\label{lem:ind:2}
Let \,$f \in \mathrm{Homeo}(\RR)$\, with \,$\mathrm{Lip}(f-Id)\leq \frac{1}{8}$\,.
Suppose that there exist \,$p\in\RR-\mathrm{Fix}(f)$\,
and a increasing sequence of positive integer numbers \,$(n_{k})_{k\geq1}$\, such that
\,$f^{n_{k}}(p) \rightarrow p$\,.
Suppose further that \,$\gamma_{k}\subset \Gamma^{f}_{\!\!p,n_{k}}$\,
is a simple closed curve.
Then $\gamma_{k}$ has no fixed points of $f$ and
the compact set of fixed points of \,$f$\,
in the interior of the disc bounded by \,$\gamma_{k}$\,
has index $1$ for $f$ when $k>>0$.
\end{lem}

We can provide versions of the results of this paper
for the sphere \,$\mathbb{S}^{2} \subset \mathbb{R}^{3}$.
We can adapt the proofs of Theorem 1.1 of \cite{dff02} and Bonatti's Main Theorem of \cite{bo01}
to the $\epsilon$-Lipschitz with respect to the identity setting.
We remind the reader that if \,$f\in \mathrm{Homeo}(\Ss)$\, is $\epsilon$-Lipschitz
with respect to the identity then
\,$f$\, is $C^{0}$-close to the identity map when \,$\epsilon>0$\, is small enough.


\vglue30pt
\section{Some examples}\label{ap:c}
\vglue10pt


There exist several results in the literature showing
existence and localization of fixed points of abelian groups.
In this section we show that the scope of our results is wider.
More precisely, we provide examples of non-abelian nilpotent groups
satisfying the conditions in the main theorem,
i.e. $\sigma$-step nilpotent subgroups of homeomorphisms that are
$\epsilon$-Lipschitz with respect to the identity, and admitting
a global fixed point.

\vskip20pt

\section*{\small  Examples arising from dimension $1$}

We build examples of nilpotent subgroups of homeomorphisms
of the plane that are inspired by examples of groups of
homeomorphisms of the real line or the circle.

Plante and Thurston proved that the $C^{2}$-regularity imposes strong restrictions on
the nilpotent groups of diffeomorphisms of the real line \cite{pt01}.


\vskip10pt
\begin{teor01}
Every nilpotent subgroup of  \,$\mathrm{Diff}^{2}([\,0\,,1\,])$\, or
\,$\mathrm{Diff}^{2}([\,0\,,1\,))$\, is abelian.
\end{teor01}


Farb-Franks generalize the previous result for \,$\mathrm{Diff}^{2}(S^{1})$\, in \cite{ff01}.
Moreover they prove the following version for diffeomorphisms of the line.


\vskip10pt
\begin{teor02}
The groups of diffeomorphisms of the line satisfy$\,:$
\begin{itemize}
\item
There exist \,$\sigma$-step nilpotent subgroups of \,$\mathrm{Diff}^{\infty}(\R)$\, for any \,$\sigma \geq 0 \,;$

\item
Any nilpotent subgroup of \,$\mathrm{Diff}^{2}(\R)$\, is metabelian, i.e.  \,$G_{(1)}$\, is
abelian$\,;$

\item
If \,$G$\, is a nilpotent subgroup of \,$\mathrm{Diff}^{2}(\R)$\,
and if any element of \,$G$\, has fixed points, then \,$G$\, is abelian.

\end{itemize}

\end{teor02}
\vskip5pt


In contrast, the case \,$C^{1}$ is radically different.


\vskip10pt
\begin{teor02}
Let  \,$M=\R\ \,, \,S^{1}$\, or \,$[\,0\,,1\,]$\,.
Every finitely generated torsion-free nilpotent group is isomorphic
to a subgroup of \,$\mathrm{Diff}^{1}(M)$\,.
\end{teor02}
\vskip5pt


Given a group \,$G$\, we denote by \,$\mathrm{Tor}(G)$\, the subset of \,$G$\, of elements of
finite order. In general \,$\mathrm{Tor}(G)$\, is not a group but it is
a normal subgroup of \,$G$\, if \,$G$\, is a nilpotent group (cf. \cite{Karga}[Theorem 16.2.7]).
We say that \,$G$\, is \,{\it torsion-free}\, if \,$\mathrm{Tor}(G)=\{Id\}$.

We denote by \,$\mathcal{N}_{n}$\, the subgroup of \,$\mathrm{GL}(n,{\mathbb Z})$\, of
lower triangular matrices such that all coefficients in the main diagonal are equal to \,$1$.
The group \,$\mathcal{N}_{n}$\, is \,$(n-1)$-step nilpotent for any \,$n \geq 2$.
Given
\,$1\leq i< n$\, we denote by \,$\eta_{i}$\, the matrix in \,$\mathcal{N}_{n}$\, such that
the coefficient in the location \,$(i+1,i)$\, is equal to \,$1$\, and all other coefficients
outside of the main diagonal vanish.
The family \,$\{\eta_{i}\}_{1\leq i<n}$\, is a generator set of \,$\mathcal{N}_{n}$\,.
A theorem of Malcev (cf. \cite{ra01}) guarantees that a
finitely generated torsion-free nilpotent group is isomorphic to a
subgroup of \,$\mathcal{N}_{n}$\, for some \,$n \geq 1$\,.

Theorem 2.13 in  \cite{ff01} and its proof imply the following result.


\vskip10pt
\begin{teor03}
Let \,$\epsilon>0$. There exists an injective homomorphism of groups
$$\psi:\mathcal{N}_{n} \longrightarrow \mathrm{Diff}^{1}([\,0\,,1\,])$$
such that$\,:$
\begin{itemize}
\item
any element of \,$\psi(\mathcal{N}_{n})$\, has derivative equal to $1$ in both
$0$ and $1 \,;$

\item
$|\psi(\eta_{i})'(x)-1|<\epsilon$\, in \,$[\,0\,,1\,]$\, for all \,$1\leq i<n$\,.

\end{itemize}

\end{teor03}
\vskip5pt


Let us define a monomorphism \,$\Psi:\mathcal{N}_{n} \longrightarrow \mathrm{Diff}^{1}(\RR)$.
Given \,$\eta \in \mathcal{N}_{n}$\, let \,$\Psi(\eta)$\, be the map defined by
$$\Psi(\eta)(t(x,y)):=\begin{cases}
[\psi(\eta)(t)](x,y)     & \textrm{if} \quad 0 \leq t \leq 1  \\
\  (x,y)   &   \textrm{if} \quad t>1  \,
\end{cases}$$
where \,$(x,y) \in {\mathbb S}^{1}$\, and \,$t \geq 0$.
Moreover, we can suppose that \,$\psi(\eta_{i})'$\, is arbitrarily
close to the constant function \,$1$\, for any \,$1 \leq i \leq n-1$.
The nilpotent group \,$\Psi(\mathcal{N}_{n})$\, is generated by the family
\,${\{ \Psi(\eta_{i}) \}}_{1 \leq i <n}$\,, whose elements are
$C^{1}$-close to the identity.
Therefore any finitely generated torsion-free nilpotent group can be realized
as a subgroup of \,$\mathrm{Diff}^{1}(\RR)$\, whose generators are arbitraly
close to the identity in the $C^{1}$-topology.

Navas proves the following result in a recent paper (cf. \cite[Th\'eor\`eme A]{navas01}).

\vskip10pt
\begin{teor04}
Let \,$G$\, be a finitely generated subgroup of sub-exponential growth of
orientation-preserving homeomorphisms of \,$[\,0\,,1]$\, or \,$S^{1}$.
Then, given any \,$\epsilon>0$\,, there exist subgroups that are topologically conjugated to
\,$G$\, and such that the generators and their inverses are \,$e^{\epsilon}$-Lipschitz homeomorphisms.
\end{teor04}
\vskip5pt

The above theorem implies that in the one dimensional setting the Lipschitz
property can be assumed for the study of the topological dynamics
of finitely generated nilpotent groups of homeomorphisms.
Finitely generated nilpotent groups have polynomial growth.
A simple calculation shows that the generators in the above theorem  are
$(e^{\epsilon}-1)$-Lipschitz with respect to the identity.

\vskip20pt

\section*{\small Real analytic examples}

Our goal is providing examples of groups of real analytic diffeomorphisms
of the plane that have a global fixed point, any nilpotency class
and generators arbitrarily close to the identity map in the $C^{1}$-topology.
Indeed we find nilpotent Lie algebras of real analytic vector fields in \,${\mathbb S}^{2}$\,
whose set of common singular points is the set \,$\{0\,,\infty\}$\, for any nilpotency class.
Our examples are finitely generated subgroups of the image by the exponential map
of such Lie algebras.
Notice that all the above examples are essentially one dimensional and share the constraints
of the one dimensional theory. For instance the previous section does not provide
an example of
a non-abelian nilpotent subgroup \,$G$\, of  \,$\mathrm{Diff}^{2}({\mathbb R}^{2})$\,
with a global fixed point since such groups do not exist in dimension \,$1$.
We show that such two dimensional examples exist for any nilpotency
class and real analytic regularity.

We obtain other interesting results. It is well known that
Lie algebras of real analytic vector fields defined in surfaces
are metabelian. It was not known if there are examples of
Lie algebras of real analytic vector fields in ${\mathbb S}^{2}$
of any nilpotency class.
We prove that this is the case and along the way we give examples of
torsion-free nilpotent groups of real analytic diffeomorphisms in the sphere
for any nilpotency class.

We denote by \,$\mathrm{Diff}^{\omega}(S)$\, the group of real analytic diffeomorphisms
defined in a real analytic manifold \,$S$.
Let \,$\mathrm{Diff}_{+}^{\omega}(S)$\, be the subgroup of \,$\mathrm{Diff}^{\omega}(S)$\,
of orientation-preserving diffeomorphisms.

\vglue5pt

Let \,${\mathfrak g}$\, be a Lie algebra.  We denote
\,${\mathfrak g}_{(0)}=\mathfrak g$\, and  let
\,${\mathfrak g}_{(j+1)}$\, be the Lie algebra generated by
\,$\big\{ [f,g] \ ; \  f \in {\mathfrak g} \ \mathrm{and} \ g \in {\mathfrak g}_{(j)} \big\}$\,
for \,$j \geq 0$. We say that \,${\mathfrak g}$\, is \,$\sigma$-{\it step nilpotent}\, if
\,$\sigma$\, is the first element in \,${\mathbb Z}_{\geq0}$\, such that
\,${\mathfrak g}_{(\sigma)}=\{0\}$. In that case we say that \,$\sigma$\, is the
\,{\it nilpotency class}\, of
\,${\mathfrak g}$. We say that \,${\mathfrak g}$\, is \,{\it nilpotent}\, if
it is $\sigma$-step nilpotent for some \,$\sigma \in {\mathbb Z}_{\geq0}$\,.

\vglue5pt

A Lie algebra \,${\mathfrak g}$\, of real analytic vector fields in a surface is
always metabelian,
i.e.  \,${\mathfrak g}_{(1)}$\, is abelian.
Moreover the nilpotent subgroups of \,$\mathrm{Diff}^{\omega}({\mathbb S}^{2})$\,
are metabelian by a theorem of Ghys \cite{Ghys-identite}. In spite of this, the nilpotency class is
not bounded. The dihedral group
\[ D_{2^{\sigma}} = \langle \, f,g \ \, ; \  f^{2^{\sigma}} = 1 \ \ \text{and} \ \
g \circ f \circ g^{-1} = f^{-1} \rangle \]%
is a $\sigma$-step nilpotent 
group acting by Mobius transformations on the sphere.
The subgroup \,$\langle f \rangle$\, is an index \,$2$\,  abelian normal subgroup of
\,$D_{2^{\sigma}}$. An analogous property always holds true for nilpotent
groups of $C^{1}$-diffeomorphisms that preserve both area and orientation.

\begin{theo}
{\bf(Franks-Handel \cite{Fr-Hand:dis})}
Let \,$G$\, be a nilpotent subgroup of \,$\mathrm{Diff}_{+}^{1}({\mathbb S}^{2})$.
Let \,$\mu$\, be a \,$\phi$-invariant Borel probability measure for any \,$\phi \in G$.
Suppose that the support of \,$\mu$\, is the whole sphere. Then either \,$G$\, is abelian
or it contains an index \,$2$\, normal abelian subgroup.
\end{theo}

In particular the above theorem implies that a torsion-free nilpotent subgroup
of  \,$\mathrm{Diff}^{1}_{+}({\mathbb S}^{2})$\, preserving a measure of total support
is always abelian. We consider groups
of real analytic diffeomorphisms instead of the area preserving hypothesis.
We are replacing a rigidity condition with another one and it is natural to ask
if analyticity restricts the examples of nilpotent groups as much as preservation of
area.
Indeed Ghys suggests in \cite{Ghys-identite} that the quotient \,$G/\mathrm{Tor}(G)$\,
is likely to be of nilpotency class less or equal than \,$2$\, for any
nilpotent subgroup \,$G$\, of \,$\mathrm{Diff}^{\omega}({\mathbb S}^{2})$.
We will show that this is not the case.

\begin{theo}
\label{teo:ntor}
Given  \,$\sigma \in {\mathbb Z}^{+}$\, there exists a  \,$\sigma$-step nilpotent
torsion-free subgroup of \,$\mathrm{Diff}_{+}^{\omega}({\mathbb S}^{2})$.
\end{theo}

We also obtain\,:
\begin{theo}
\label{teo:c1pid}
Given  \,$\sigma \in {\mathbb Z}^{+}$\, there exist
\,$\phi_{1},\ldots ,\phi_{\sigma+1} \in \mathrm{Diff}_{+}^{\omega}({\mathbb R}^{2})$\,
sharing a common fixed point and such that
\,$\langle \phi_{1}, \hdots, \phi_{\sigma+1} \rangle$\,  is $\sigma$-step nilpotent.
Moreover, we can choose the generators \,$\phi_{j}$\, arbitrarily and  uniformly close to
the identity in the $C^{1}$-topology.
\end{theo}

\vskip10pt
\subsection*{A method to obtain nilpotent Lie algebras}

Let us explain our method for the real plane.
We denote by \,${\mathbb R}[x,y]$\, the ring of real polynomials in two variables.
Let \,$\alpha_{1}\,,\beta_{1}$\, be quotients of convenient elements of \,${\mathbb R}[x,y]$\, such that\,:

\begin{enumerate}
\item
$d \alpha_{1} \wedge d \beta_{1} = D \, dx \wedge dy$\, where $1/D \in {\mathbb R}[x,y]$\,;

\item
$\frac{1}{D}  d \alpha_{1}$\, and \,$\frac{1}{D}d \beta_{1}$\, are $1$-forms with no poles in ${\mathbb R}^{2}$.
\end{enumerate}
Let us consider vector fields \,$X_{1}$\, and \,$Y_{1}$\, defined in \,${\mathbb R}^{2}$\, such that
\[
\left\{ \begin{array}{ccc}
X_{1}(\alpha_{1}) & \equiv & 0 \\
Y_{1}(\alpha_{1}) & \equiv & 1
\end{array} \right.
\ \ \text{and} \ \
\left\{ \begin{array}{ccl}
X_{1}(\beta_{1}) & \equiv & 1 \\
Y_{1}(\beta_{1}) & \equiv & 0\,.
\end{array} \right.
\]
A straightforward calculation implies
\[ X_{1} = \frac{-\frac{\partial \alpha_{1}}{\partial y}}{D} \frac{\partial}{\partial x} +
\frac{\frac{\partial \alpha_{1}}{\partial x}}{D} \frac{\partial}{\partial y} \quad \text{and} \quad
Y_{1} = \frac{\frac{\partial \beta_{1}}{\partial y}}{D} \frac{\partial}{\partial x} +
\frac{-\frac{\partial \beta_{1}}{\partial x}}{D} \frac{\partial}{\partial y} .
\]

Condition (2) implies that $\frac{1}{D}  d \alpha_{1}$\, is polynomial.
Since $\frac{1}{D}  d \alpha_{1}$\, is equal to
$\frac{1}{D} \frac{\partial \alpha_{1}}{\partial x} dx + \frac{1}{D} \frac{\partial \alpha_{1}}{\partial y} dy$
then \,$X_{1}$\, is a polynomial vector field. Analogously
\,$Y_{1}$\, is a polynomial vector field.
Moreover \,$\alpha_{1}$\, and \,$\beta_{1}$\, can be considered as variables outside of
a proper real algebraic set. We obtain \,$[X_{1},Y_{1}] \equiv 0$\, since
\,$[X_{1},Y_{1}](\alpha_{1}) \equiv [X_{1},Y_{1}](\beta_{1}) \equiv 0$\, by definition.

 Therefore the real vector space
\,$\langle X_{1}, \alpha_{1} X_{1}, \hdots, \alpha_{1}^{l-1} X_{1}, Y_{1}\rangle_{\mathbb R}$\,
is a \,$l$-step nilpotent Lie algebra of polynomial vector fields of \,${\mathbb R}^{2}$\,  if
\,$\alpha_{1}^{l-1} X$\, has no poles.

Let us introduce the choice of \,$\alpha_{1}$\, and \,$\beta_{1}$\, that is going to provide our
examples. Consider
\[ \left( \alpha_{1}, \beta_{1}, d \alpha_{1} \wedge d \beta_{1} \right) =
\left(
\frac{1}{(x^{2}+y^{2})^{k}} \,, \frac{y}{x (x^{2}+y^{2})^{p}} \,, -2 k \frac{dx \wedge dy}{x^{2} (x^{2}+y^{2})^{k+p}}
\right)   \ \ \text{with} \ \ k\hskip1pt,p\geq1\,. \]%

\noindent
We obtain
\[ X_{1} = x^{2} (x^{2}+y^{2})^{p-1}\left\{ -y \frac{\partial}{\partial x} + x \frac{\partial}{\partial y} \right\}
 \]
and
\[ Y_{1}= - \frac{1}{2k} (x^{2}+y^{2})^{k-1} \left\{
x\big(x^{2} + (1-2p) y^{2}\big) \frac{\partial}{\partial x} + y\big((1+2p) x^{2} + y^{2}\big) \frac{\partial}{\partial y}
\right\} . \]

\noindent
 We also have that \,$(d \alpha_{1})/D$\,
and \,$(d \beta_{1})/D$\, has no poles.
Moreover if \,$p-1 \geq k(l-1)$\, then
the vector field \,$\alpha_{1}^{l-1} X_{1}$\, has no poles.

 In summary,
\,$\langle X_{1}, \alpha_{1} X_{1}, \hdots, \alpha_{1}^{l-1} X_{1}, Y_{1}\rangle_{\mathbb R}$\, is a
$l$-step nilpotent
Lie algebra of polynomial vector fields
if  \,$p-1 \geq k(l-1)$.

\vskip10pt
\subsection*{Examples of nilpotent groups  on the sphere}

Let us try to generalize the example to the sphere. Unfortunately the vector fields
\,$X_{1}$\, and \,$Y_{1}$\, do not extend to real analytic vector fields of \,${\mathbb S}^{2}$.

Let us study the dynamics of \,$X_{1}$\, and \,$Y_{1}$.
Since \,$X_{1}(\alpha_{1}) \equiv 0$\, the function
\,$x^{2}+y^{2}$\, is a first integral of \,$X_{1}$. The trajectories of \,$X_{1}$\, are contained in
circles and \,$\mathrm{Sing (X_{1})} = \{x=0\}$.
On the other hand the trajectories of \,$Y_{1}$\, are transversal to the level curves of
\,$x^{2}+y^{2}$\, since \,$Y_{1} (\alpha_{1}) \equiv 1$.
This condition also implies that \,$\mathrm{exp}(t Y_{1})$\, sends
\,$\{\alpha_{1} =s \}$\, to \,$\{\alpha_{1} =s+t \}$\, for \,$s\,,t \in {\mathbb R}$.

Consider an annulus
\[ A = \{z \in {\mathbb C} \ \, ; \, R^{-1} < |z| <R \} = \{ (x,y) \in {\mathbb R}^{2}
\ \, ; \  R^{-1} < \sqrt{x^{2}+y^{2}} < R \} \]
for some \,$R >1$.
We define \,$X_{2} = - (1/z)_{*} X_{1}$\, and \,$Y_{2} = - (1/z)_{*} Y_{1}$\,; both vector fields
are defined in  \,${\mathbb S}^{2} - \{ |z| \leq R^{-1} \}$.
We claim that there exists a real analytic diffeomorphism
\,$\phi: A \to A$\, such that \,$\phi_{*} X_{1} = X_{2}$\, and \,$\phi_{*} Y_{1} = Y_{2}$.
More precisely \,$\phi$\, conjugates the pair of functions
\,$(\alpha_{1}, \beta_{1})$\, with
\[ \left( - \alpha_{1} \circ \frac{1}{z} + \kappa \,, - \beta_{1}  \circ \frac{1}{z} \right)=
\left(
- (x^{2} + y^{2})^{k} +\kappa \,, \frac{y (x^{2} + y^{2})^{p}}{x}
\right) \]
where \,$\kappa=R^{2 k} + R^{-2 k}$. The sets
\,$\alpha_{1}(A)$\, and \,$(- \alpha_{1} \circ (1/z) + \kappa)(A)$\, coincide
by definition of \,$\kappa$.
The map \,$(\alpha_{1},\beta_{1})$\, is not injective since
\,$(\alpha_{1},\beta_{1})(x,y) = (\alpha_{1},\beta_{1})(-x,-y)$.
Anyway a homeomorphism \,$\phi: \overline{A} \to \overline{A}$\, is
uniquely determined by the conditions
\[ \left( - \alpha_{1} \circ \left( \frac{1}{z} \right) + \kappa \right) \circ \phi \equiv \alpha_{1}
 \quad  ; \quad
- \beta_{1} \circ \left( \frac{1}{z} \right) \circ \phi \equiv \beta_{1} \quad \text{and} \quad
\phi \{x>0\} \subset \{x>0\}\]%
Since \,$\alpha_{1}$\,, $\beta_{1}$\, are
coordinates
in the neighborhood of any point in \,$x \neq 0$\,  then the map
$$\phi: A - \{x=0\} \to A - \{x=0\}$$
is a real analytic
diffeomorphism. Analogously \,$\alpha_{1}$\, and \,$1/\beta_{1}$\, are real analytic
coordinates in the neighborhood of the points in the line \,$x=0$.
We deduce that \,$\phi:A \to A$\, is a real analytic diffeomorphism.

Let \,$S$\, be the real analytic surface \,$S$\, obtained by pasting the charts
\,$U_{1} = \{|z| < R\}$\, and \,$U_{2} = \{|z| > R^{-1}\}$\, by the
diffeomorphism \,$\phi$.
The surface \,$S$\, is homeomorphic to a sphere, so it is real analytically
diffeomorphic to a sphere.
Since \,$\phi_{*} X_{1} =X_{2}$\, there exists a unique
real analytic vector field \,$X$\, in \,${\mathbb S}^{2}$\, such that
\,$X_{|U_{1}} \equiv X_{1}$\, and \,$X_{|U_{2}} \equiv X_{2}$.
Analogously we can define \,$Y$\,, $\alpha$\, and \,$\beta$\,
such that \,$Y_{|U_{1}} \equiv Y_{1}$\,, $Y_{|U_{2}} \equiv Y_{2}$\,,
\,$\alpha_{|U_{1}} \equiv \alpha_{1}$\,,
\,$\alpha_{|U_{2}} \equiv - \alpha_{1} \circ (1/z) + \kappa$\,,
\,$\beta_{|U_{1}} \equiv \beta_{1}$\, and
\,$\beta_{|U_{2}} \equiv - \beta_{1} \circ (1/z)$.
We obtain that
\[ {\mathcal G} :{=} \langle X, \alpha X, \hdots, \alpha^{l-1} X, Y \rangle_{\mathbb R} \]
is a $l$-step nilpotent Lie algebra of real analytic vector fields of
the sphere.
In particular \,${\mathcal G}$\, is non-abelian if \,$l >1$.

The vector field \,$Y$\, has two singularities corresponding to the origin of both local charts.
We can suppose \,$\mathrm{Sing}(Y) = \{\,0\,,\infty \}$.
The common singular set of all vector fields of the form \,$P(\alpha) X$\,, where \,$P$\,
is a polynomial of degree less than \,$l$, is a circle \,$C$.
The equation of \,$C$\, is \,$x=0$\, in both charts.
Hence all vector fields in \,${\mathcal G}$\, are singular at both \,$0$\, and \,$\infty$.
We obtain that \,${\mathcal G}$\, is a Lie algebra of real analytic vector fields defined in both
\,${\mathbb R}^{2}$\, and \,${\mathbb S}^{2}$.

\vskip10pt
\subsection*{The Baker-Campbell-Hausdorff formula}

Given a Lie algebra \,${\mathcal G}$\, as defined above the set
\,$G=\mathrm{exp}({\mathcal G})$\, of time \,$1$\, flows of elements of \,${\mathcal G}$\,
is a subset of \,$\mathrm{Diff}_{+}^{\omega}({\mathbb S}^{2})$.
In this section we show that \,$G$\, is a nilpotent group of the same
nilpotency class as \,${\mathcal G}$. This is an application of
Baker-Campbell-Hausdorff formula. The material in this section
is well-known, it is included for the sake of completeness.

Let \,$G$\, be a Lie group with Lie algebra \,${\mathfrak g}$.
The exponential \,$\mathrm{exp}({\mathfrak g})$\, does not coincide in general with the
connected component of the identity but anyway it contains a neighborhood of the identity element.
Hence given elements \,$\mathrm{exp}(Z)$\, and \,$\mathrm{exp}(W)$\, in \,$G$\, closed to the identity element
we have that
the element \,$\mathrm{exp}(Z) \mathrm{exp}(W)$\, belongs to \,$\mathrm{exp}({\mathfrak g})$\,
and \,$\log (\mathrm{exp}(Z) \mathrm{exp}(W))$\, is given by the following formula
due to Dynkin:
\begin{equation}
\label{equ:BCH}
 \sum_{n>0}\frac {(-1)^{n-1}}{n} \sum_{ \begin{smallmatrix} {r_i + s_i > 0} \\ {1\le i \le n} \end{smallmatrix}}
\frac{\big(\sum_{i=1}^n (r_i+s_i)\big)^{-1}}{r_1!s_1!\cdots r_n!s_n!} \, [ Z^{r_1} W^{s_1}  \ldots Z^{r_n} W^{s_n} ]
\end{equation}
where \,$[Z^{r_1} W^{s_1} \ldots Z^{r_n} W^{s_n} ]$\, is equal to
\[
[ \underbrace{Z,\ldots[Z}_{r_1} , \,\ldots\, [ \underbrace{W,\ldots [W}_{s_1}, \ldots,
[ \underbrace{Z,[Z,\ldots[Z}_{r_n} ,[ \underbrace{W,[W,\ldots W}_{s_n} ]]\ldots]]. \]
The Baker-Cambpell-Hausdorff theorem says
that such a universal formula exists
and that \,$\log (\mathrm{exp}(Z) \mathrm{exp}(W)) - (Z+W)$\, is a bracket polynomial in
\,$Z$\, and \,$W$.
Moreover if \,$G$\, is a simply connected nilpotent Lie group
then the exponential map \,$\mathrm{exp}: {\mathfrak g} \to G$\, is an analytic diffeomorphism
(cf. \cite{Corwin-Greenleaf}[Theorem 1.2.1]).
In such a case the sum defining Equation (\ref{equ:BCH}) contains finitely many
non-zero terms.

Let us explain how to apply these ideas to the study of the elements of the
group \,$\mathrm{Diff}({\mathbb R}^{2},0)$\, of real analytic germs of diffeomorphism
defined in a neighborhood of \,$(0\,,0)$.
Any element \,$\varphi \in \mathrm{Diff}({\mathbb R}^{2},0)$\, has a power series
expansion of the form
\[ \varphi(x,y) = \left( \sum_{j+k \geq 1} a_{j,k}\, x^{j} y^{k}, \sum_{j+k \geq 1} b_{j,k}\, x^{j} y^{k} \right) \]
where \,$\sum_{j+k \geq 1} a_{j,k}\, x^{j} y^{k}$\, and \,$\sum_{j+k \geq 1} b_{j,k}\, x^{j} y^{k}$\,
are convergent power series with real coefficients and
\,$(a_{1,0}\, x + a_{0,1} y \,, b_{1,0}\, x + b_{0,1} y)$\, is a linear isomorphism.
In fact the previous conditions on
\,$\sum_{j+k \geq 1} a_{j,k}\, x^{j} y^{k}$\, and \,$\sum_{j+k \geq 1} b_{j,k}\, x^{j} y^{k}$\,
determine a unique element of \,$\mathrm{Diff}({\mathbb R}^{2},0)$\, by the inverse function theorem.
Let \,${\mathfrak m}$\, be the maximal ideal of the local ring \,${\mathbb R}\{x,y\}$.
We define the order \,$\nu (\varphi)$\, of contact with the identity as
\[ \nu (\phi) = \max \bigg\{ l \in {\mathbb N} \ \ ;  \sum_{j+k \geq 1} a_{j,k}\, x^{j} y^{k} - x \in {\mathfrak m}^{l}
\quad \text{and} \quad \sum_{j+k \geq 1} b_{j,k}\, x^{j} y^{k} - y \in {\mathfrak m}^{l} \bigg\}.  \]
We define the group
\,$\mathrm{Diff}_{1}({\mathbb R}^{2},0)=\{ \varphi \in \mathrm{Diff}({\mathbb R}^{2},0) \  ; \
\nu(\varphi) \geq 2\}$\,
of tangent to the identity elements. Consider
the group \,$j^{k} \mathrm{Diff}_{1}({\mathbb R}^{2},0)$\, of $k$-jets of tangent to the identity
elements for \,$k \in {\mathbb Z}^{+}$. It  is the subgroup of
\,$\mathrm{GL}({\mathfrak m}/{\mathfrak m}^{k+1}, {\mathbb R})$\, defined by the action
by composition of \,$\mathrm{Diff}_{1}({\mathbb R}^{2},0)$\, in \,${\mathfrak m}/{\mathfrak m}^{k+1}$\, where
a map \,$\varphi \in \mathrm{Diff}_{1}({\mathbb R}^{2},0)$\, induces a linear map
\[ \begin{array}{ccc}
{\mathfrak m}/{\mathfrak m}^{k+1} & \longrightarrow &  {\mathfrak m}/{\mathfrak m}^{k+1} \\
g + {\mathfrak m}^{k+1} & \longmapsto & g \circ \varphi + {\mathfrak m}^{k+1}.
\end{array}
\]
The group \,$j^{k} \mathrm{Diff}_{1}({\mathbb R}^{2},0)$\,
is a  contractible matrix algebraic Lie group composed of unipotent elements for any
\,$k \in {\mathbb Z}^{+}$. It is nilpotent since
\,$\nu([\varphi, \eta]) > \max \{ \nu(\varphi), \nu(\eta)\}$\, for all
\,$\varphi, \eta \in \mathrm{Diff}_{1}({\mathbb R}^{2},0)$.
Thus we can apply Equation (\ref{equ:BCH}) in
\,$j^{k} \mathrm{Diff}_{1}({\mathbb R}^{2},0)$\, and it extends to
\,$\mathrm{Diff}_{1}({\mathbb R}^{2},0)$\,
since \,$\mathrm{Diff}_{1}({\mathbb R}^{2},0)$ is the projective limit
$\lim_{\leftarrow} j^{k} \mathrm{Diff}_{1}({\mathbb R}^{2},0)$.
There is a subtility in this construction since in general
Formula (\ref{equ:BCH}) does not define a convergent power series and
it is then necessary to define the exponential of a formal vector field
via Taylor's formula
\begin{equation}
\label{equ:exp}
\mathrm{exp} (Z) = (x \circ \mathrm{exp} (Z), y \circ \mathrm{exp} (Z)) =
\left( x + \sum_{j=1}^{\infty} \frac{Z^{j} (x)}{j!} \,, y + \sum_{j=1}^{\infty} \frac{Z^{j} (y)}{j!} \right)
\end{equation}
where \,$Z$\, is understood as an operator on functions and
\,$Z^{j}$\, is the \,$j$th iterate of \,$Z$.
The convergence of the Baker-Campbell-Hausdorff is not a problem in the following since in all
subsequent applications of
Formula (\ref{equ:BCH}) the sum is finite.


\begin{proof}[Proof of Theorem \ref{teo:ntor}]
Consider \,$l=\sigma$\, in the construction of \,${\mathcal G}$.
We define \,$G = \mathrm{exp}({\mathcal G})$.
It is a subset of \,$\mathrm{Diff}_{+}^{\omega}({\mathbb S}^{2})$\,
by compactness of \,${\mathbb S}^{2}$.
Since \,$0 \in \mathrm{Fix}(G)$\, and the vanishing order of any vector field in
\,${\mathcal G}$\, at \,$(0\,,0)$\, is higher than \,$2$\,
then we can consider \,$G$\, as a subset of
\,$\{ \varphi \in \mathrm{Diff}_{1}({\mathbb R}^{2},0) \ ; \ \nu (\varphi) \geq 3 \}$\,
by Equation (\ref{equ:exp}).
The use of Equation (\ref{equ:BCH})
implies that since \,${\mathcal G}$\, is a Lie algebra then \,$G$\, is a group.
In particular \,$G$\, is a subgroup of both \,$\mathrm{Diff}_{1}({\mathbb R}^{2},0)$\,
and the group of real analytic diffeomorphisms of the sphere.
The torsion-free nature of \,$G$\, is a consequence of the analogous property
for \,$\mathrm{Diff}_{1}({\mathbb R}^{2},0)$\,
(indeed \,$\nu (\phi) = \nu (\phi^{k})$\, for all \,$\phi \in
\mathrm{Diff}_{1}({\mathbb R}^{2},0) - \{Id\}$\,
and \,$k \in {\mathbb Z}^{*}$).

There are no non-trivial elements in \,$G$\, with arbitrarily high order of
contact with the identity; otherwise the analogous property is satisfied
for \,${\mathcal G}$\, by Equation (\ref{equ:exp}) and this is impossible since
\,${\mathcal G}$\, is finite dimensional.
Therefore \,$G$\, can be interpreted as a subgroup of
\,$j^{k} \mathrm{Diff}_{1}({\mathbb R}^{2},0)$\, for some \,$k \in {\mathbb Z}^{+}$\,
big enough.
There exists an equivalence of categories between the finite dimensional
nilpotent Lie algebras and the unipotent affine algebraic groups over fields
of characteristic \,$0$\,
(cf. \cite{Demazure-Gabriel}[IV, \S \ 2, n\textordmasculine \ 4, Corollaire 4.5]).
Therefore the group \,$j^{k} G$\, induced by \,$G$\, in \,$j^{k} \mathrm{Diff}_{1}({\mathbb R}^{2},0)$\,
is a connected affine algebraic group for any \,$k \in {\mathbb Z}^{+}$.
In this context the nilpotency class of the group and its Lie algebra coincide
(cf. \cite{Demazure-Gabriel}[IV, \S \ 4, n\textordmasculine \ 1, Corollaire 1.6]).
We deduce that \,$G$\, is a $\sigma$-step nilpotent group.
\end{proof}

\begin{remark}
It is easy to show by hand that \,$\langle X, \hdots, \alpha^{l-j-1} X \rangle_{\mathbb R}$\,
is the Lie algebra of
\,$G_{(j)}$\, for any \,$0 < j <l$\, and \,$G_{(l)}=\{Id\}$.
\end{remark}

\begin{proof}[Proof of Theorem \ref{teo:c1pid}]
Consider \,$l=\sigma$\, in the construction of \,${\mathcal G}$.
We define the group
\[ J =
\langle \mathrm{exp}(t X),  \mathrm{exp}(t \alpha X), \hdots,  \mathrm{exp}(t \alpha^{\sigma-1} X),  \mathrm{exp}(t Y) \rangle \]
for some fixed \,$t \in {\mathbb R}^{*}$\,  small.
Analogously as in the proof of Theorem \ref{teo:ntor}
we can consider that \,$G$\, and \,$J$\, are subgroups of \,$j^{k} \mathrm{Diff}_{1}({\mathbb R}^{2},0)$\,
for some \,$k \in {\mathbb Z^{+}}$\, big enough.
Let \,$H$\, be the smallest algebraic group in \,$j^{k} \mathrm{Diff}_{1}({\mathbb R}^{2},0)$\,
containing \,$J$.
The group \,$H$\, is a connected unipotent algebraic group contained in \,$G$.
The Lie algebra \,${\mathfrak h}$\, of \,$H$\, coincides with \,${\mathcal G}$\, by construction.
As a consequence the groups \,$G$\, and \,$H$\, also coincide.
The group \,$J$\,  is $\sigma$-step nilpotent since its algebraic closure is.
We can obtain $\sigma$-step nilpotent subgroups of \,$\mathrm{Diff}_{+}^{\omega}({\mathbb R}^{2},0)$\,
sharing \,$(0\,,0)$\, as a fixed point and with generators arbitrarily and uniformly close to the
identity map by Proposition \ref{pro:conv1} below.
\end{proof}

The next result completes the proof of Theorem \ref{teo:c1pid}.
\begin{pro}
\label{pro:conv1}
Let \,$Z \in {\mathcal G}$.
The diffeomorphism \,$\mathrm{exp}(t Z)$\,
converges to the identity map in the strong $C^{1}$-topology when
\,$t \to 0$.
\end{pro}
\begin{proof}
The result is obviously true for the $C^{1}$-topology for diffeomorphisms
of the sphere. We will show that the result still holds true in
\,$\mathrm{Diff}^{1}({\mathbb R}^{2},0)$.

Let \,$\eta_{t} = \mathrm{exp}(t Z)$. The diffeomorphism \,$\eta_{t}$\, converges
to the identity map in any compact set of the plane.
Let us study the properties of the one parameter flow \,$\eta_{t}$\,
in the neighborhood of \,$\infty$.
Let \,$z$\, be a complex coordinate in the Riemann sphere.
We consider the coordinate \,$w=1/z$\, in order to study the
behavior of the diffeomorphisms in the neighborhood of the point
\,$z=\infty$.
The vector field
\,$\partial /\partial z$\, is equal to \,$-w^{2} \partial /\partial w$.
We obtain
\[ \frac{\partial}{\partial x} = - (\hat{x}^{2} - \hat{y}^{2}) \frac{\partial}{\partial \hat{x}}
-  2 \,\hat{x} \, \hat{y} \frac{\partial}{\partial \hat{y}} \quad \text{and} \quad
\frac{\partial}{\partial y} = 2 \,\hat{x} \,\hat{y} \frac{\partial}{\partial \hat{x}}
- (\hat{x}^{2} - \hat{y}^{2}) \frac{\partial}{\partial \hat{y}} \]
where \,$z=x+i y$\, and \,$w = \hat{x} + i \hat{y}$.
The vanishing order of \,$Z$\, at \,$\infty$\, is higher than \,$2$.
The expression \,$\hat{\eta}_{t} = (1/z) \circ \eta \circ (1/w)$\,
of \,$\eta_{t}$\, satisfies
$$\hat{\eta}_{t}(w) = w + t \! \! \sum_{j+k \geq 3} c_{j,k}(t)\, \hat{x}^{j} \, \hat{y}^{k}$$%
in the  coordenate \,$w$\,
where \,$c_{j,k}$\, is a polynomial with complex coefficients for \,$j + k \geq 3$.
Since
\[ \eta_{t}(z)- z = \eta_{t} \left( \frac{1}{w} \right) - \frac{1}{w}  =
\frac{1}{\hat{\eta}_{t}(w)} - \frac{1}{w} =
\frac{1}{w + t O(w^{3}) } -  \frac{1}{w} = O(t w) \]
then \,$\eta_{t}$\, converges to the identity map in the neighborhood of \,$\infty$\, in the
$C^{0}$-topology.
The expression
\[ \frac{\partial}{\partial x} (\eta_{t}(z) - z) =
\left( - (\hat{x}^{2} - \hat{y}^{2}) \frac{\partial}{\partial \hat{x}}
-  2 \hat{x} \hat{y} \frac{\partial}{\partial \hat{y}}\right)
\left( \eta_{t} \left( \frac{1}{w} \right) - \frac{1}{w} \right) = \]
\[ \left( - (\hat{x}^{2} - \hat{y}^{2}) \frac{\partial}{\partial \hat{x}}
-  2 \hat{x} \hat{y} \frac{\partial}{\partial \hat{y}}\right)
\left(  \frac{1}{\hat{\eta}_{t}(w)} - \frac{1}{w} \right)=
O(t w^{2}) \]
and the analogue for \,$(\partial /\partial y)(\eta_{t}(z) - z)$\, imply that
\,$\mathrm{exp}(tZ)$\, converges to \,$Id$\, in the neighborhood of \,$\infty$\,
in the $C^{1}$-topology when \,$t \to 0$.
\end{proof}

\begin{remark}
It is easy to check out that \,$\lim_{t \to 0} \mathrm{exp}(t Z) = Id$\, in the strong
$C^{k}$-topology for any \,$k \in {\mathbb Z}^{+}$.
\end{remark}


\vglue30pt

\end{document}